\documentclass[12pt]{article}
\usepackage{amsmath}
\usepackage{amsfonts,amssymb}
\usepackage{extarrows}
\usepackage{geometry}
\usepackage{authblk} 
\usepackage{hyperref} 

\usepackage[utf8]{inputenc}
\def\0{\emptyset}

\newtheorem{theorem}{Theorem}[section]
\newtheorem{definition}[theorem]{Definition}
\newtheorem{lemma}[theorem]{Lemma}

\newtheorem{cor}[theorem]{Corollary}
\newtheorem{proposition}[theorem]{Proposition}
\newtheorem{remark}[theorem]{Remark}

\newenvironment{proof}{{\noindent\it Proof.}\quad}{\hfill $\square$\par}

\usepackage{graphicx}

\usepackage{enumerate}

\usepackage{
	amsmath,			
	amssymb,			
	enumerate,		    
	graphicx,			
	lastpage,			
	multicol,			
	multirow,			
	pifont,			    
}

\usepackage[numbers]{natbib}

\begin{document}


\title{Treewidth of generalized Hamming graph, bipartite Kneser graph and generalized Petersen graph}
\author[1]{\small\bf Yichen Wang\thanks{E-mail:  wangyich22@mails.tsinghua.edu.cn}}
\author[2]{\small\bf Mengyu Cao\thanks{E-mail:  myucao@ruc.edu.cn}}
\author[1]{\small\bf Zequn Lv\thanks{E-mail:  lvzq19@mails.tsinghua.edu.cn}}
\author[1]{\small\bf Mei Lu\thanks{E-mail:  lumei@tsinghua.edu.cn}}


\affil[1]{\small Department of Mathematical Sciences, Tsinghua University, Beijing 100084, China}
\affil[2]{\small Institute for Mathematical Sciences, Renmin University of China, Beijing 100086, China}

\date{}

\maketitle\baselineskip 16.3pt

\begin{abstract}
Let $t,q$ and $n$ be positive integers. Write $[q] = \{1,2,\ldots,q\}$.
The generalized Hamming graph $H(t,q,n)$ is the graph whose vertex set is
the cartesian product of $n$ copies of $[q]$
$(q\ge  2)$, where two vertices are adjacent if  their Hamming distance is at most $t$.
In particular, $H(1,q,n)$ is the well-known Hamming graph and $H(1,2,n)$ is the hypercube.
In 2006, Chandran and Kavitha described the asymptotic value of $tw(H(1,q,n))$, where $tw(G)$ denotes the treewidth of  $G$.
In this paper, we give the exact pathwidth of $H(t,2,n)$ and show that $tw(H(t,q,n)) = \Theta(tq^n/\sqrt{n})$ when $n$ goes to infinity.
Based on those results, we show that the treewidth of the bipartite Kneser graph $BK(n,k)$ is $\binom{n}{k} - 1$ when $n$ is sufficient large relative to $k$ and  the bounds of $tw(BK(2k+1,k))$ are given.
Moreover, we present the bounds of the treewidth of generalized Petersen graph.
\end{abstract}


\vskip.3cm

\section{Introduction}

Treewidth is a well-studied parameter in graph theory.
Many NP-complete problem can be solved in polynomial time on graphs of bounded treewidth~\cite{intro_tw_algorithm_edgecoloring_zhou1996edge,intro_tw_algorithm_listcoloriong_bruhn2016list,intro_tw_algorithm_steinertree_CHIMANI201267}.
Besides, treewidth is also useful in structural graph theory. For example, Robertson and Seymour used it to prove the Graph Minor Theorem~\cite{GraphMinor_1_robertson1983graph,GraphMinor_3_robertson1984graph,GraphMinor_2_robertson1986graph}.
In the past few decades, there has been much literature investigating the treewidth of certain graphs (see for example,~\cite{tw_Kneser_Harvey2014,intro_tw_certaingraph_linegraph_harvey2018treewidth,intro_tw_certaingraph_circlegraph_kloks1993treewidth,intro_tw_certaingraph_geometric_li2017treewidth,intro_tw_certaingraph_productsofconnected_wood2013treewidth}).
However, it is difficult to estimate the treewidth even asymptotically.

Throughout this paper, graphs are finite, simple and undirected.
Let $V(G)$ and $E(G)$ denote the vertex set and edge set of $G$, respectively.
The degree of a vertex $v \in V(G)$ in $G$ is denoted by $d_G(v)$, and we ignore the subscript $G$ in case of no ambiguity.
The maximum and minimum degree of $G$ are denoted as $\Delta(G)$ and $\delta(G)$, respectively.
Write $[n] = \{1,2,\ldots, n\}$ and denote by $\binom{[n]}{k}$ the family of all $k$-subsets of $[n]$.
The treewidth, pathwidth and bandwidth of graph $G$ are denoted by $tw(G)$, $pw(G)$ and $bw(G)$, respectively.
It is worth mentioning that the isoperimetric problem of the hypercube is an important topic which provides a foundation for our subsequent proof. The readers may refer to~\cite{pw_Hypercube_harper1966optimal, bandwidth_q2inf_harper1999isoperimetric,rashtchian2022edge,hart1976note,beltran2023sharp} for more details.

The \textit{Hamming graph} $H(q,n)$ is the graph on $q^{n}$ vertices, which correspond to all $n$-vectors whose components are from $[q]$ ($q\ge 2$).
When $q=2$, it is more common to treat each vertex of $H(2,n)$ as a binary $n$-vector, that is, an $n$-vector whose components are from $\{0,1\}$.
Two of the vertices in $H(q,n)$ are adjacent only when they differ in just one coordinate.
The \textit{hypercube graph} is a special case of the Hamming graph when $q=2$ which is well-studied in parallel computing, coding theory and many other areas\cite{ aiello1991coding, hayes1989hypercube, jimenez2002hypercube}.

The Hamming distance of two $n$-vectors is the number of coordinates where one differs from the other.
Two vertices in a Hamming graph are adjacent if and only if their Hamming distance is no more than $1$.
If two vertices in $H(q,n)$ is adjacent if and only if their Hamming distance is no more than $t$, then we denote the \textit{generalized Hamming graph} as $H(t,q,n)$.
Then $H(1,q,n)=H(q,n)$ and $H(1,2,n)$ is the hypercube graph. We call $H(t,2,n)$ the \textit{generalized hypercube graph} and we will give the exact value of its pathwidth and bandwidth as Theorem~\ref{thm: bandwidth of generalized Hypercube}. When $t=1$, our result covers the previous results of hypercube graph. It is interesting to see that in the generalized form of hypercube graph, its treewidth still can be described by an exact expression. Having exact form of treewidth is not a common phenomenon, especially with such a complicated form.

\begin{theorem}\label{thm: bandwidth of generalized Hypercube}We have
	\begin{equation}\label{eq: bandwidth of generalized Hypercube}
	\begin{aligned}
		&pw(H(t,2,n)) = bw(H(t,2,n)) \\
		=& \sum_{k = \left\lfloor (n-t)/2 \right\rfloor }^{\left\lfloor (n-t)/2 \right\rfloor +t -1} \binom{n}{k} + \sum_{a=0}^{\left\lfloor (n-t-1)/2 \right\rfloor } \left( \binom{t+2a}{t+a-1} - \binom{t+2a}{a-1} \right).
	\end{aligned}
	\end{equation}
\end{theorem}

The treewidth of $H(q,n)$ is asymptotically $\Theta(q^{n}/\sqrt{n})$ from~\cite{tw_pw_HammingGraph_chandran2006treewidth}. To be more specific, there exists constants $c_1$ and $c_2$ not depending on $q$, such that for sufficiently large $n$, $c_1 q^{n}/\sqrt{n} \le tw(H(q,n)) \le c_2 q^{n}/\sqrt{n}$.
In this paper, we generalize this result to generalized Hamming graph $H(t,q,n)$ as Theorem~\ref{thm: asymptotic tw(H(t,q,n))}.
The previous result about $tw(H(q,n))$ can be derived from our result by letting $t=1$. 
It is interesting to see that the distance variable $t$ in the generalized Hamming graph $H(t,q,n)$ has linear impact on its treewidth.

\begin{theorem}\label{thm: asymptotic tw(H(t,q,n))}
	There exists constant $c_1$ and $c_2$ not depending on $t$ or $q$, such that for any positive integers $t$ and $q$. When $n$ is sufficiently large, we have that
	\[
		c_1tq^n / \sqrt{n} \le tw(H(t,q,n)) \le c_2tq^n / \sqrt{n}.
	\]
\end{theorem}

\textit{Kneser graph} $K(n,k)$ is the graph whose vertex set is $\binom{[n]}{k}$ and where two vertices are adjacent if and only if the two corresponding sets are disjoint.
The treewidth of Kneser graph is studied by Harvey in 2014~\cite{tw_Kneser_Harvey2014}.
\textit{Generalized Kneser graph}, \textit{$q$-Kneser graph} and \textit{generalized $q$-Kneser graph} are three derived class from Kneser graph and their treewidth can also be exactly described when $n$ is large enough~\cite{tw_qKneser_cao2021treewidth, intro_tw_certaingraph_Generalized_Kneser_cmy2022,metsch2024treewidth}.
\textit{Bipartite Kneser graph} is also an important variant of Kneser graph whose vertex set is $\binom{[n]}{k}$ and $\binom{[n]}{n-k}$ denoted by $BK(n,k)$ where $2k \le n$.
A $k$-subset $A$ and a $(n-k)$-subset $B$ in bipartite Kneser graph is adjacent if $A \subseteq B$.
We will give the exact value of $tw(BK(n,k))$ when $n$ is sufficient large relative to $k$ as Theorem~\ref{thm: tw of BK(n,k) when n is large enough}.
When $n=2k+1$, we give the bounds of the treewidth of the bipartite Kneser graph as Theorem~\ref{thm: tw BK J uplow}.

\begin{theorem}\label{thm: tw of BK(n,k) when n is large enough}
	If $3\binom{n-k}{k} \ge 2\binom{n}{k}$ and $k \ge 2$, then $tw(BK(n,k)) = \binom{n}{k} - 1$.
\end{theorem}

\begin{theorem}\label{thm: tw BK J uplow}
	There exists two constants $c_1$ and $c_2$ such that for any positive integer $k$,
\[
	c_1\frac{1}{k}\binom{2k+1}{k} \le tw(BK(2k+1, k)) \le tw(J(2k+1,k)) \le c_2\binom{2k+1}{k}.
\]
\end{theorem}

\textit{Petersen graph} is a well-studied Kneser graph $K(5,2)$.
\textit{Generalized Petersen graph} is an extension of Petersen graph denoted by $G_{n,k}$ whose vertex set and edge set are
\begin{equation*}	
\begin{aligned}
	V(G_{n,k})&=\{v_1,\hdots,v_n,u_1,\hdots,u_n\},\\ E(G_{n,k})&={\{v_i u_i\}} \cup {\{v_i v_{i+1}\}} \cup {\{u_i u_{i+k}\}}, i=1,2,\hdots,n,
\end{aligned}
\end{equation*}
where subscripts are to be read modulo $n$ and $k < n/2$.
We will show the treewidth of $G_{n,k}$ when $n$ is sufficient large relative to $k$ as Theorem~\ref{thm: tw generalized Petersen}

\begin{theorem}\label{thm: tw generalized Petersen}
	Let $n$ and $k$ be positive integers satisfying that $n\geq 8{(2k+2)}^2$. Then we have
	\[ 
	2k+1 \le tw(G_{n,k}) \le pw(G_{n,k}) \le 2k+2.
	\]
\end{theorem}

Theorem~\ref{thm: tw of BK(n,k) when n is large enough},~\ref{thm: tw BK J uplow} and~\ref{thm: tw generalized Petersen} can be viewed as different generalizations of Kneser graph. And in Theorem~\ref{thm: tw BK J uplow}, it is interesting to see that the treewidth of the bipartite Kneser graph has a close relationship with that of the well-studied Johnson graph. The treewidth of Johnson graph $J(n,k)$ is also an interesting topic.
When $k=2$, $tw(J(n,2))$ and $pw(J(n,2))$ have exact formulas while for other $k$ it remains unknown~\cite{fabila2024treewidth}.
Our result may help to have a deeper understanding of the treewidth of Johnson graph.

The rest of this paper is organized as follows.
In Section~\ref{section: treewidth of generalized Hamming graph}, we give a proof of the treewidth of generalized Hamming graph~(Theorem~\ref{thm: bandwidth of generalized Hypercube} and Theorem~\ref{thm: asymptotic tw(H(t,q,n))}).
In this part, we mainly use properties of bandwidth and techniques of Hales numbering and the isoperimetric problems by Harper~\cite{pw_Hypercube_harper1966optimal,bandwidth_q2inf_harper1999isoperimetric}.

In Section~\ref{section: treewidth of bipartite Kneser and Johnson}, we study the treewidth of the bipartite Kneser graph and Johnson graph~(Theorem~\ref{thm: tw of BK(n,k) when n is large enough} and Theorem~\ref{thm: tw BK J uplow}).
In this part, we mainly use the techniques of separators, properties of cross-intersecting families and chordal completions. Since the Johnson graph can be viewed as a slice of generalized Hamming graph, we need the results in Section~\ref{section: treewidth of generalized Hamming graph} to prove our results.

In Section~\ref{section: tw generalized Petersen graph}, we study the treewidth of generalized Petersen graph~(Theorem~\ref{thm: tw generalized Petersen}).
In this part, we mainly use brambles and path-decomposition constructions.

\section{Preliminaries and definitions}

In this section, we give  definitions involving in treewidth, pathwidth and bandwidth of a graph $G(V,E)$.  

\vskip.2cm
\begin{definition}\label{def: treewidth}
A tree-decomposition of a graph $G(V, E)$ is a pair $(X, T)$,
where $T(I, F)$ is a tree with vertex set $I$ and edge set $F$, and $X = \{X_i \mid i \in I\}$ is a family of subsets of $V$, one for each node of $T$, such that:
\begin{itemize}
	\item $\bigcup \limits_{i \in I}X_i = V$.
	\item For each edge $uv \in E$, there exists an $i \in I$ such that $u,v \in X_i$.
	\item For all $i,j,k \in I$, if $j$ is on the path from $i$ to $k$ in $T$, then $X_i \cap X_k \subset X_j$.
\end{itemize}
\end{definition}
The width of a tree-decomposition $(X, T)$ is $\max \limits_{i\in I} |X_i| - 1$.
The treewidth of a graph $G$ is the minimum treewidth over all possible tree-decompositions of $G$ and denoted by $tw(G)$.
The problem of deciding whether a graph has tree decomposition of treewidth at most $k$ is NP-complete~\cite{tw_complexity_arnborg1987complexity}.
However, there is an exact algorithm finding treewidth of given graph $G$ when taking $tw(G)$ as a constant~\cite{tw_exact_algo_1_bodlaender1993linear, tw_exact_algo_2_korhonen2023improved}.
A path decomposition of $G$ is a tree decomposition $(X,T)$ in which $T$ is required to be a path.
The pathwidth of $G$ is defined to be the minimum width over all path decompositions of $G$ and is denoted by $pw(G)$.

A bijection $\phi~:~ V \rightarrow \{1,2,\ldots,n\}$ is called an ordering of the vertices of $G$ (in short, an ordering of $G$).
Then for any edge $e = \{u,v\} \in E$, let $\Delta(e,\phi) = |\phi(u) - \phi(v)|$.

\vspace{.2cm}
\begin{definition}
	A bandwidth of a graph $G(V,E)$, denoted by $bw(G)$, is the minimum over all possible orderings $\phi$ of $V$ of the maximum value of $\Delta(e,\phi)$ over all edges $e \in E$.
	That is,
	\[
		bw(G) = \min_{\phi} \max_{e \in E} \Delta(e, \phi).
	\]
\end{definition}
There are important inequalities between treewidth, pathwidth and bandwidth as following.

\vspace{.2cm}
\begin{proposition}[\cite{bandwidth_relation_bodlaender1998partial}]
	For any graph $G$,
	\begin{equation}
		\label{eq: tw pw bw inequality}
		tw(G) \le pw(G) \le bw(G).
	\end{equation}
\end{proposition}

Let $X \subseteq V(G)$ be a subset of vertices and $G[X]$ be the subgraph induced in $G$ by $X$. Define $G-X = G[V(G)-X]$.
Given $p \in (0,1)$, define the \emph{$p$-separator} of $G$ to be a subset $X \subseteq V(G)$ such that
no  component of $G-X$ contains more than $p|V(G)-X|$ vertices.
Proposition~\ref{prop: tw and separator relation} describes the relationship between treewidth and separator.

\vspace{.2cm}
\begin{proposition}[\cite{GraphMinor_2_robertson1986graph}]\label{prop: tw and separator relation}
	For any graph $G$, there exists a $1/2$-separator of $G$ with at most $tw(G) +1$ vertices.
\end{proposition}

Corollary~\ref{coro: tw and separator 2/3} is directly from Proposition~\ref{prop: tw and separator relation}.

\vspace{.2cm}
\begin{cor}\label{coro: tw and separator 2/3}
	For any graph $G$, there exists a separator $X$ of $G$ with at most $tw(G) +1$ vertices.
	And there exists a partition of $V(G)-X$ into sets $A$ and $B$ such that
	\[
		|V(G)-X|/3 \le |A|,|B| \le 	2|V(G)-X|/3.
	\]
\end{cor}

Proposition~\ref{prop: tw and separator relation} and Corollary~\ref{coro: tw and separator 2/3} are useful tools to estimate the lower bound of treewidth.

\section{Treewidth of generalized Hamming graph}\label{section: treewidth of generalized Hamming graph}

\subsection{Bandwidth of Hamming graph \texorpdfstring{$H(t,2,n)$}{H(t,2,n)}}

The pathwidth and bandwidth of hypercubes can be exactly calculated as following.

\vspace{.2cm}
\begin{proposition}[\cite{tw_pw_HammingGraph_chandran2006treewidth}]\label{prop: bandwidth of Hypercube}
	We have
\begin{equation}\label{eq: bandwidth of Hypercube}
	pw(H(1,2,n)) = bw(H(1,2,n)) = \sum_{m=0}^{n-1} \binom{m}{ \left\lfloor m/2 \right\rfloor }.
\end{equation}
\end{proposition}


In this section, we intend to prove Theorem~\ref{thm: bandwidth of generalized Hypercube}.
When $t=1$, Eq~\ref{eq: bandwidth of generalized Hypercube} is exactly the same as Eq~\ref{eq: bandwidth of Hypercube}. Thus, we can view Theorem~\ref{thm: bandwidth of generalized Hypercube} as a generalization of Proposition~\ref{prop: bandwidth of Hypercube}. Using the following techniques, we will derive some recursion formulas~(Proposition~\ref{prop: bandwidth recursion formula} and Proposition~\ref{prop: rank recursion formula}) and use induction to prove Theorem~\ref{thm: bandwidth of generalized Hypercube}. 
However, Eq~\ref{eq: bandwidth of generalized Hypercube} is much more complicated and once we know its formula, it is always ``easy'' to verify it when we have some recursion formulas using induction.
The exact expression of Eq~\ref{eq: bandwidth recursion formula} and~\ref{eq: rank recursion formula} actually come from our elegant observation from some instances with computer assistance.


To prove Theorem~\ref{thm: bandwidth of generalized Hypercube}, we need some preparation.
In~\cite{pw_Hypercube_harper1966optimal}, Harper showed that if an ordering $\varphi$ of $G(V,E)$ is in Hales order (i.e., a Hales numbering), then $\max \limits_{e \in E} \Delta(e,\varphi)$ takes minimum over all numbering, that is, $bw(G) = \min \limits_{\phi} \max \limits_{e\in E}\Delta(e,\phi) = \max \limits_{e\in E}\Delta(e,\varphi)$.

\begin{definition}[Hales numbering~\cite{pw_Hypercube_harper1966optimal}]
If there exists an ordering $\phi$ all of whose beginning segments obey the following two conditions, then we call such orderings \textit{Hales numberings}. Note that Hales numbering does not always exists and it is not unique.
\begin{enumerate}
	\item For a set of $l$ vertices, let $\Phi_l$ be the number of vertices in the set having neighbors not in the set. $\Phi_l$ must be minimized for all beginning segments $S_l = \{v \in V\mid \phi(v) \le l \}$.
	\item The $\Phi_l' = l - \Phi_l$ ``interior vertices'' of $S_l$ must be numbered $1,2,\ldots, \Phi_l'$, i.e., have the lowest possible numbers on them. 
\end{enumerate}
\end{definition}

Harper~\cite{pw_Hypercube_harper1966optimal} also give a sample of Hales numbering of $H(t,2,n)$\footnote{In~\cite{pw_Hypercube_harper1966optimal}, Harper did not prove this statement. Harper first gave a Hales numbering of hypercube and, then claimed that the numbering is also in Hales order for the distance generalized graph (that is, $H(t,2,n)$) in comments~(III (b)).}.
From~\cite{pw_Hypercube_calculation_wang2009explicit}, we can build such Hales numbering $\varphi$ in the following way.

Define an $\binom{n}{k}\times n$ matrix $A_{k}^{(n)}$ as following.
Let $A_{n}^{(n)}$ be $(\underbrace{1,1,\ldots,1}_{\mbox{$n$ factors}})$ and let $A_{0}^{(n)}$ be $(\underbrace{0,0,\ldots,0}_{\mbox{$n$ factors}})$ for any positive integer $n$.
When $0 < k < n$, $A_{k}^{(n)}$ is defined recursively as Eq~\ref{eq: recursive definition of A_k^n}.

\begin{equation}\label{eq: recursive definition of A_k^n}
	A_{k}^{(n)} \triangleq  \left[ \begin{array}{cc}
		A_{k-1}^{(n-1)} & \mathbf{1} \\
		A_{k}^{(n-1)} & \mathbf{0} \\
	\end{array} \right],
\end{equation}
where $\mathbf{1}$ (resp. $\mathbf{0}$) is an all one (resp.\ all zero) $\binom{n-1}{k-1 }$ (resp. $\binom{n-1}{k}$) column vector.
Clearly, the rows of $A_k^{(n)}$ enumerates all binary vectors of length $n$ whose number of `$1$' is $k$.
Define $S^{(n)}$ as following:

\begin{equation}
S^{(n)} \triangleq \left[ \begin{array}{c}
	A_0^{(n)} \\
	A_1^{(n)} \\
	\vdots \\
	A_n^{(n)} \\
\end{array} \right].
\end{equation}
Then
$S^{(n)}$ is an $2^n \times n$ matrix.
Each row of $S^{(n)}$ excatly corresponds to an $n$-vector (a vertices of $H(t,2,n)$) and vice versa.
Let $\eta^{(n)}$ be the ordering of $H(t,2,n)$ defined by row order of $S^{(n)}$, that is, if $v \in V(H(t,2,n))$ corresponds to the $i$-th row of $S^{(n)}$, then let $\eta^{(n)}(v) = i$.

 $\eta^{(n)}$ is a Hales numbering of $H(t,2,n)$ (see~\cite{pw_Hypercube_calculation_wang2009explicit} and~\cite{pw_Hypercube_harper1966optimal}) and, hence, we have the following Proposition~\ref{prop: hales ordering}.

\begin{proposition}\label{prop: hales ordering}
	$bw(H(t,2,n)) = \max \limits_{e \in E}\Delta(e, \eta^{(n)})$.
\end{proposition}

In order to calculate $bw(H(t,2,n))$, we need some more definitions.

\begin{definition}
Given a graph $G(V,E)$, for two vertex subsets $V_1 \subseteq V$ and $V_2 \subseteq V$ numbered by ordering $\eta_1$ and $\eta_2$ respectively, the adjacency matrix of $V_1$ and $V_2$ is defined to be a $|V_1| \times |V_2|$ matrix $M$ such that for any $u \in V_1$ and $v \in V_2$,
\begin{equation*}
	M(\eta_1(u), \eta_2(v)) = \left\{ \begin{array}{cc}
		1 & \mbox{if $\{u,v\} \in E$}, \\
		0 & \mbox{otherwise}.
	\end{array} \right.
\end{equation*}	
The adjacent matrix of  $G(V,E)$ with ordering $\eta$ is the adjacent matrix of $V$ and $V$ ordered by $\eta$ as defined above.
\end{definition}

\vspace{.2cm}
\begin{definition}
	The bandwidth of an $s \times s$ non-zero square matrix $M$ denoted by $bw(M)$ is the maximum absolute value of the difference between the row and column indices of a nonzero element of that matrix, i.e.,
	\begin{equation}
		bw(M) = \max \limits_{1 \le i \le s, 1 \le j \le s} \{ |i-j| \mid M(i,j) \neq 0 \},
	\end{equation}
	which is exactly the maximum Manhattan distance from a nonzero element to the main diagonal of the matrix (see Figure~\ref{fig: bandwidth square}).
\end{definition}

\vspace{.2cm}
\begin{remark}\label{rmk: bandwidth}
	Given an ordering $\eta$ of a graph $G$, $\max \limits_{e \in E}\Delta(e,\eta)$ is equal to the bandwidth of the adjacency matrix of $G$ ordered by $\eta$.
\end{remark}

\begin{figure}[t]
	\centering         
	\begin{minipage}{0.33\linewidth}
		\centering         
		\includegraphics[width=\linewidth]{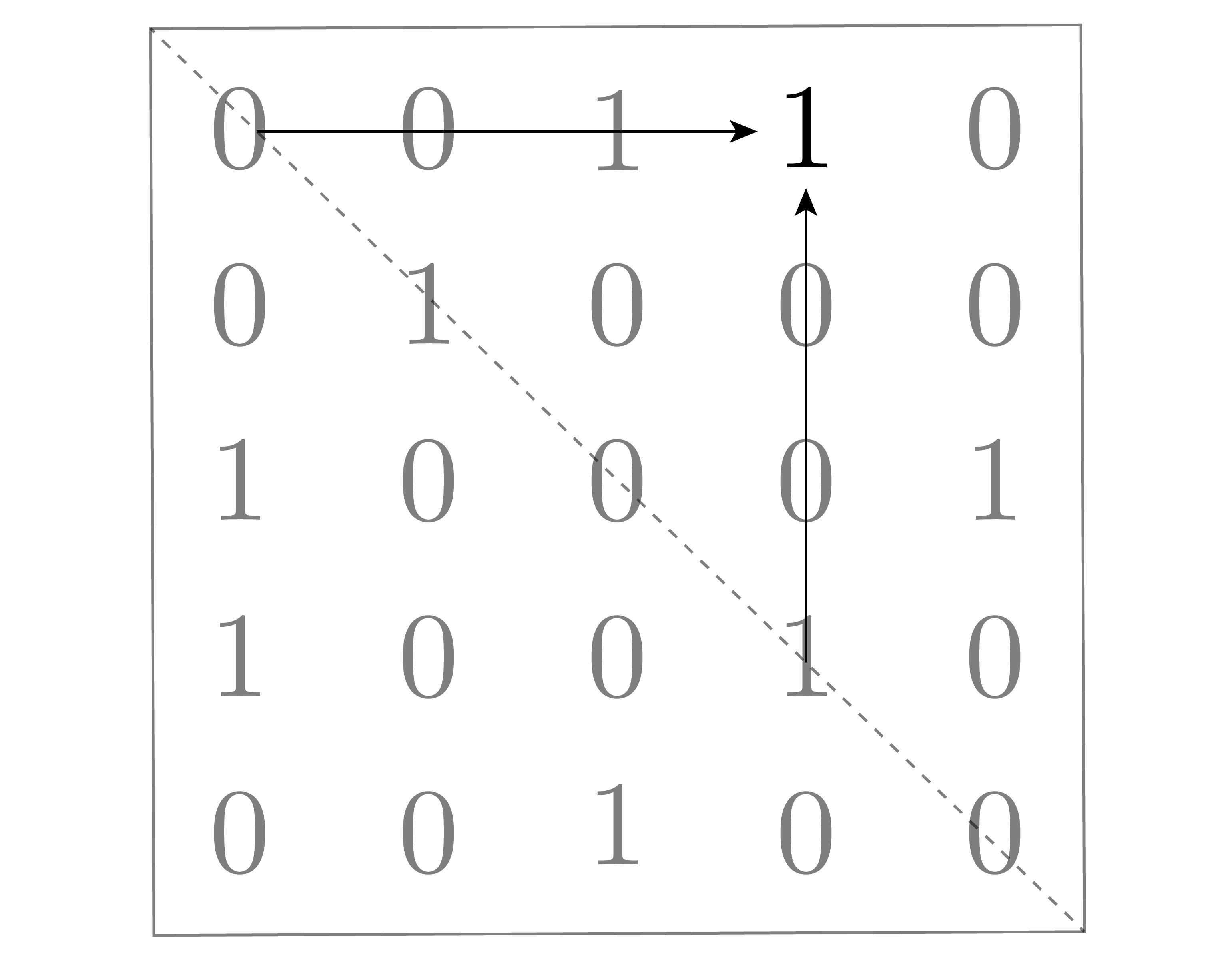}
		\caption{Bandwidth of a square matrix.}\label{fig: bandwidth square}
	\end{minipage}
	\begin{minipage}{0.59\linewidth}
		\vspace{-0.6cm}
		\centering         
		\includegraphics[width=\linewidth]{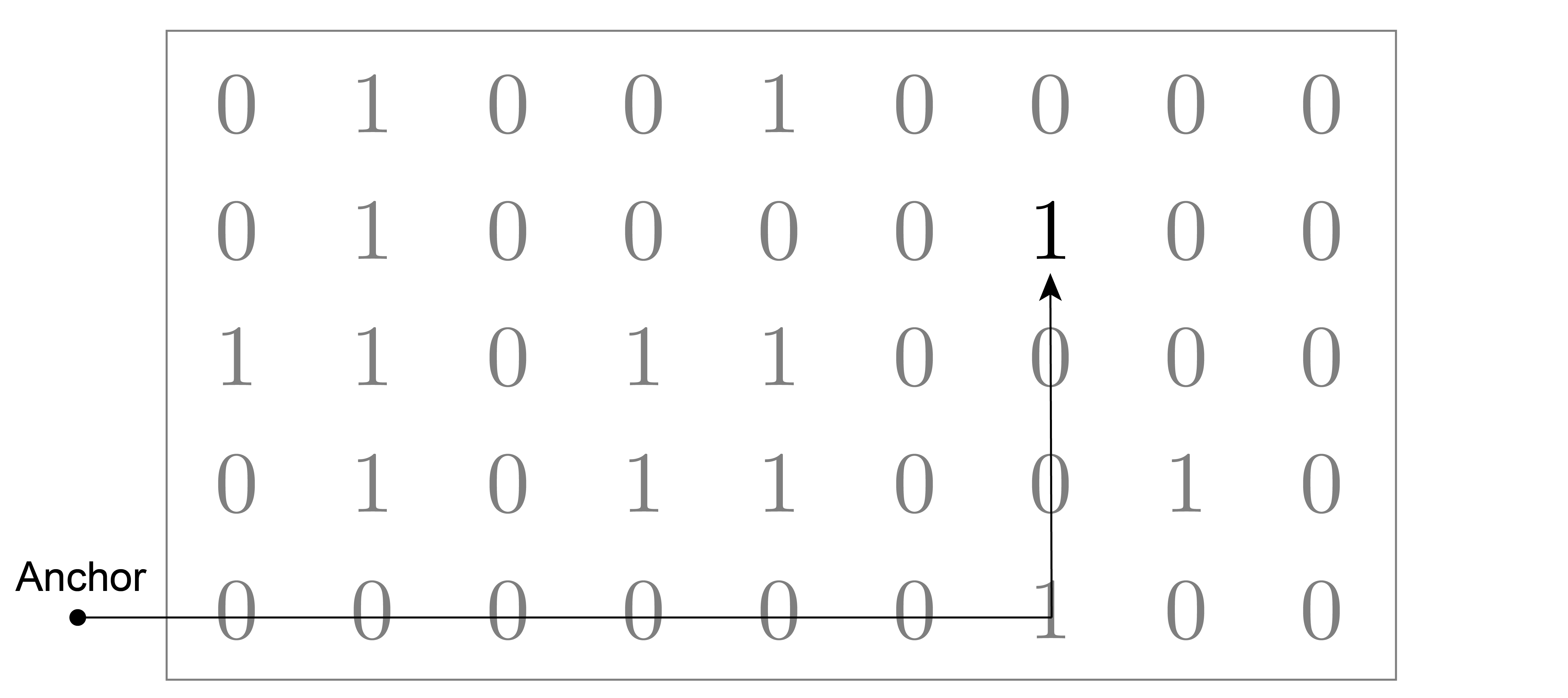}
		\caption{Manhattan radius and the imaginary anchor.}\label{fig: anchor}
	\end{minipage}
\end{figure}

Remark~\ref{rmk: bandwidth} is directly from definitions.
Let $M^{(t,n)}$ be the adjacency matrix of $H(t,2,n)$ ordered by $\eta^{(n)}$. From Remark~\ref{rmk: bandwidth} and Proposition~\ref{prop: hales ordering}, we have Lemma~\ref{lemma: bandwidth graph adjacent matrix}.

\vspace{.2cm}
\begin{lemma}\label{lemma: bandwidth graph adjacent matrix}
	$bw(H(t,2,n)) = bw(M^{(t,n)})$.
\end{lemma}

Now our aim is to give exact value of $bw(M^{(t,n)})$.
To achieve this, we need to define the Manhattan radius of matrix.

\vspace{.2cm}
\begin{definition}
	For an $s \times t$ non-zero matrix $M$, its Manhattan radius $r(M)$ is defined by
	\begin{equation}
		r(M) = \max \limits_{1 \le i \le s, 1 \le j \le t} \{ s - i + j \mid M(i,j) \neq 0 \},
	\end{equation}
which is exactly the maximum Manhattan distance from a nonzero element of $M$ to the position immediately to the left of the bottom-left corner of $M$ (an imaginary element $M(s,0)$).
	The imaginary element is called the anchor of $M$ (see Figure~\ref{fig: anchor}).
	For the convenience of the proof, let $r(M) = -\infty$ if $M$ is a zero matrix or  an empty matrix.

\end{definition}

Note that $r(M) > 0$ if and only if $M$ is non-zero and non-empty. If $M$ is an $s \times s$ symmetric matrix, then by definition,
\begin{equation}\label{eq: radius bandwidth relation}
r(M) = bw(M) + s.
\end{equation}

Let $V_k^{(t,n)} \subseteq V(H(t,2,n))$ be a vertex set containing vertices of $H(t,2,n)$ whose corresponding $n$-vector has exactly $k$ ones, where $0\le k\le n$. Then $\{V_0^{(t,n)}, V_1^{(t,n)}, \ldots, V_n^{(t,n)}\}$ forms a partition of  $V(H(t,2,n))$.
Recall that each row of $A^{(n)}_k$ correspond to a vertex of $V_k^{(t,n)}$ and vice versa.
Let $\eta^{(n)}_{k}=\eta^{(n)}|_{V_k^{(t,n)}}$.
Let $M_{k,k'}^{(t,n)}$ be the adjacent matrix of $V_k^{(t,n)}$ and $V_{k'}^{(t,n)}$ ordered by $\eta^{(n)}_{k}$ and $\eta^{(n)}_{k'}$ respectively.

For the convenience of proof, let $M^{(t,n)}_{k,k'}$ be empty matrix if either $k$ or $k'$ is larger than $n$ or less than zero.
It is easy to verify that $r(M^{(t,n)}_{k,k'}) > 0$ if and only if $|k-k'|\le t$ and $0 \le k,k' \le n$.
Then from definition, we have

\begin{equation}\label{eq: M^(t,n) expansion}
M^{(t,n)} = \left[ \begin{array}{cccccc}
	M_{0,0}^{(t,n)} & M_{0,1}^{(t,n)} & M_{0,2}^{(t,n)} & \cdots & M_{0,n-1}^{(t,n)} & M_{0,n}^{(t,n)} \\
	M_{1,0}^{(t,n)} & M_{1,1}^{(t,n)} & M_{1,2}^{(t,n)} & \cdots & M_{1,n-1}^{(t,n)} & M_{1,n}^{(t,n)} \\
	M_{2,0}^{(t,n)} & M_{2,1}^{(t,n)} & M_{2,2}^{(t,n)} & \cdots & M_{2,n-1}^{(t,n)} & M_{2,n}^{(t,n)} \\
	\vdots & \vdots & \vdots & \ddots & \vdots & \vdots \\
	M_{n-1,0}^{(t,n)} & M_{n-1,1}^{(t,n)} & M_{n-1,2}^{(t,n)} & \cdots & M_{n-1,n-1}^{(t,n)} & M_{n-1,n}^{(t,n)} \\
	M_{n,0}^{(t,n)} & M_{n,1}^{(t,n)} & M_{n,2}^{(t,n)} & \cdots & M_{n,n-1}^{(t,n)} & M_{n,n}^{(t,n)} \\
	
\end{array} \right].
\end{equation}
Note that $M^{(t,n)}$ is an $2^n\times 2^n$ matrix. If $t \ge n$, then $M^{(t,n)}$ is an all one matrix except for the elements in main diagonal. Hence $bw(M^{(t,n)}) = 2^{n} - 1$.
In the following, we only consider the situation when $t < n $.
Since the bandwidth of $M^{(t,n)}$ is the maximum Manhattan distance from a non-zero element to the main diagonal and $M^{(t,n)}$ is symmetric,
we only need to consider non-zero element from submatrices in Eq~\ref{eq: M^(t,n) expansion} on the diagonal or on the top right of the diagonal, that is,
non-zero elements from $M_{k,k'}^{(t,n)}$ with $k' \ge k$.
Since $M_{k,k'}^{(t,n)}$ is all zero when $|k - k' | > t$,
we only need to consider non-zero element from $M_{k,k'}^{(t,n)}$ with $k \le k' \le k+t$.
By the definition of bandwidth and Manhattan radius, we have Proposition~\ref{prop: bandwidth recursion formula}.

\vspace{.2cm}
\begin{proposition}\label{prop: bandwidth recursion formula} We have
	\begin{equation}
	\begin{aligned}
		bw(M^{(t,n)}) & = \max \limits_{ \begin{array}{c} k = 0, \ldots ,n-1,\\ 1 \le p \le t, \\ p + k \le n \end{array} } \left\{ \sum_{q = 1}^{p-1} \binom{n}{k+q} + r(M^{(t,n)}_{k,k+p}), bw(M^{(t,n)}_{k,k}) \right\}.
	\end{aligned}
	\label{eq: bandwidth recursion formula}
	\end{equation}

\end{proposition}
Note that $M^{(t,n)}_{k,k}$ is a symmetric matrix. By Eq~\ref{eq: radius bandwidth relation}, we have $bw(M^{(t,n)}_{k,k}) = r(M^{(t,n)}_{k,k}) + \binom{n}{k}$.
If we can calculate all $r(M_{k,k'}^{(t,n)})$ with $k \le k' \le k+t$, then we can calculate $bw(M^{(t,n)})$ via Eq~\ref{eq: bandwidth recursion formula} and, hence, obtain $bw(H(t,2,n))$ by Lemma~\ref{lemma: bandwidth graph adjacent matrix}.
Now, our aim is to calculate all $r(M_{k,k'}^{(t,n)})$ with $k \le k' \le k+t$.

From Eq~\ref{eq: recursive definition of A_k^n}, we have that if $t,n,k,p$ are non-negative intergers satisfying that $k+p \le n $ and $n \ge 1$, then

\begin{equation}\label{eq: M_k,k+p^(t,n) expansion}
    M_{k,k+p}^{(t,n)} = \left( \begin{array}{cc} M_{k-1,k+p-1}^{(t,n-1)} & M_{k-1, k+p}^{(t-1, n-1)} \\ M_{k,k+p-1}^{(t-1,n-1)} & M_{k, k+p}^{(t, n-1)} \end{array} \right).
\end{equation}
Let
$r_{t,n,k,p}^{(1)} = r(M_{k-1, k+p-1}^{(t,n-1)}) + \binom{n-1}{k}$,
$r_{t,n,k,p}^{(2)} = r(M_{k-1, k+p}^{(t-1, n-1)}) + \binom{n-1}{k} + \binom{n-1}{k+p-1}$,
$r_{t,n,k,p}^{(3)} =r(M_{k, k+p-1}^{(t-1,n-1)})$ and
$r_{t,n,k,p}^{(4)} =r(M_{k, k+p}^{(t, n-1)}) + \binom{n-1}{k+p-1}$.
From the definition of $r(M^{(t,n)}_{k,k+p})$ and Eq~\ref{eq: M_k,k+p^(t,n) expansion}, we have Proposition~\ref{prop: rank recursion formula}.

\vspace{.2cm}
\begin{proposition}\label{prop: rank recursion formula}\
	For non-negative integers $t,n,k,p$ satisfying that $k+p \le n$ and $n \ge 1$, we have
	\begin{equation}
		\begin{aligned}
			r(M_{k,k+p}^{(t,n)}) &= \max \left\{ r_{t,n,k,p}^{(1)}, r_{t,n,k,p}^{(2)}, r_{t,n,k,p}^{(3)}, r_{t,n,k,p}^{(4)}\right\}.
		\end{aligned}
		\label{eq: rank recursion formula}
	\end{equation}
\end{proposition}

Considering $M^{(t,n)}_{k_1, k_2}$ with $k_1 \le k_2 \le k_1 + t$, it is easy to verify that if the parity of $k_2-k_1$ and $t$ is different, then $M^{(t,n)}_{k_1, k_2} = M^{(t-1,n)}_{k_1, k_2}$.
Hence, we only need to calculate all $M^{(t,n)}_{k_1, k_2}$ with $k_1 \le k_2 \le k_1 + t$ and $k_2-k_1 \equiv t \pmod{2}$.
In this case, assume $t - (k_2 - k_1) = 2s$. Then $M^{(t,n)}_{k,k+t-2s}$ is non-zero.
If $t \ge n-1$, then $M^{(t,n)}_{k,k+t-2s}$ is an all-one matrix~(except the main diagonal when $t=2s$) and, thus, $r(M^{(t,n)}_{k,k+t-2s}) = \binom{n}{k} + \binom{n}{k+t-2s} - 1$. When $1 \le t < n-1$, we let

\begin{footnotesize}
\begin{equation*}
\begin{aligned}
	&A^{(1)}_{t,n,k,s} = \sum \limits_{a=0}^{k-s-1} \left( \binom{t-s+2a}{t-s+a-1} - \binom{t-s+2a}{a-1} \right),\\
	&A^{(2)}_{t,n,k,s} = \sum \limits_{m = t-3s+1+2k}^{n-s} \left( \binom{m-1}{k+t-2s-1} - \binom{m-1}{k-s-1} \right),\\
	&B_{t,n,k,s} = \sum \limits_{a=0}^{n-t-k+s-1} \left( \binom{t-s+2a}{t-s+a-1} - \binom{t-s+2a}{a-1} \right),\\
	&C_{t,n,k,s} = \sum \limits_{m= n-s+1}^{n} \binom{m-1}{k+t-2s-1}= \binom{n}{k+t-2s} - \binom{n-s}{k+t-2s},\\
	&D_{t,n,k,s} = \sum \limits_{m= k+t-2s+1}^{n} \binom{m-1}{k+t-2s-1}.
\end{aligned}
\end{equation*}
\end{footnotesize}

Here we define $\binom{x}{y} = 0$ if $y<0$ or $y>x$.
Note that all $A^{(1)}, A^{(2)}, B,C,D$ terms are non-negative.
It is easy to verify that if $k+t-2s = n $, then $D_{t,n,k,s} = 0 $, otherwise $k+t-2s < n $, then $D_{t,n,k,s} = \binom{n}{k+t-2s} - \binom{k+t-2s}{k+t-2s} = \binom{n}{k+t-2s} - 1$.
Then we have the following results.

\vspace{.2cm}
\begin{lemma}\label{lemma: guessing r^(t,n)_k1,k2}
	Let $t,n, k,s$ be non-negative integers satisfy that $t \ge 2s, n \ge 1, t \ge 1$ and $k + t -2s \le n$. Then
	\begin{equation}\label{eq: lemma: guessing r^(t,n)_k1,k2}
		\begin{aligned}
			r(M^{(t,n)}_{k,k+t-2s})  &= \left\{ \begin{array}{ll}
\binom{n}{k} + \binom{n}{k+t-2s} - 1 & \mbox{if } t\ge n-1, \\
				\binom{n}{k} + A^{(1)}_{t,n,k,s} + A^{(2)}_{t,n,k,s} + C_{t,n,k,s} & \mbox{if } t<n-1,0 \le k - s \le \left\lfloor (n-t)/2 \right\rfloor, \\
				\binom{n}{k} + B_{t,n,k,s} + C_{t,n,k,s}& \mbox{if } t<n-1, \left\lfloor (n-t)/2 \right\rfloor \leq k -s \leq n-t  , \\
				\binom{n}{k} + D_{t,n,k,s} & \mbox{if $t<n-1$, $k-s \le 0$ or $k-s \ge n-t$.}
			\end{array} \right.
		\end{aligned}
	\end{equation}
\end{lemma}

\vspace{.2cm}
\begin{lemma}\label{lemma: bandwidth guessing M^{(t,n)}}
	We have	\begin{equation}\label{eq: bandwidth guessing M^{(t,n)}}
	\begin{aligned}
		bw(M^{(t,n)}) = \sum_{k = \left\lfloor (n-t)/2 \right\rfloor }^{\left\lfloor (n-t)/2 \right\rfloor +t -1} \binom{n}{k} + \sum_{a=0}^{\left\lfloor (n-t-1)/2 \right\rfloor } \left( \binom{t+2a}{t+a-1} - \binom{t+2a}{a-1} \right).
	\end{aligned}
	\end{equation}
\end{lemma}

From Eq~\ref{eq: bandwidth recursion formula} and~\ref{eq: rank recursion formula}, we can prove Lemmas~\ref{lemma: guessing r^(t,n)_k1,k2} and~\ref{lemma: bandwidth guessing M^{(t,n)}} by induction.
The complete proofs of Lemmas~\ref{lemma: guessing r^(t,n)_k1,k2} and~\ref{lemma: bandwidth guessing M^{(t,n)}} can be find in {\bf Appendix A} and {\bf Appendix B}, respectively.
Lemmas~\ref{lemma: bandwidth guessing M^{(t,n)}} and~\ref{lemma: bandwidth graph adjacent matrix} show the exact bandwidth of $H(t,2,n)$.

For a set $S \subseteq V$, let $N(S)= \{ v \in V-S \mid \exists u \in S, uv \in E\}$, $ \varPhi(S) = |N(S)|$ and
 $b_v(l, G) = \min \limits_{S \subseteq V, |S| = l}\varPhi(S)$.
Harper~\cite{pw_Hypercube_harper1966optimal} showed that $bw(G) = \max \limits_{1 \le s \le |V|} b_v(s,G)$ if $G$ admits a Hales numbering.
Therefore, $bw(H(t,2,n)) = \max \limits_{1 \le s \le 2^n} b_v(s,H(t,2,n))$.

\vspace{.2cm}
\begin{theorem}[Theorem 1 of~\cite{tw_pw_HammingGraph_chandran2006treewidth}]\label{thm: cite pw bv relation}
	Let $G(V,E)$ be any graph on $n$ vertices, and let $1 \le s \le n$. Then $pw(G) \ge b_v(s,G)$.
\end{theorem}

\begin{lemma}\label{lemma0}
	We have	$pw(H(t,2,n))$ $ = bw(H(t,2,n))$.
\end{lemma}

\begin{proof}
By Theorem~\ref{thm: cite pw bv relation}, we have $pw(H(t,2,n)) \ge b_v(s,H(t,2,n))$ for all $1 \le s \le 2^n$.
Then $pw(H(t,2,n)) \ge \max \limits_{1 \le s \le 2^n} b_v(s,H(t,2,n)) = bw(H(t,2,n))$.
Combining Eq~\ref{eq: tw pw bw inequality}, we have $pw(H(t,2,n))$ $ = bw(H(t,2,n))$.
\end{proof}

Theorem~\ref{thm: bandwidth of generalized Hypercube} can be derived from Lemmas~\ref{lemma: bandwidth graph adjacent matrix},~\ref{lemma: bandwidth guessing M^{(t,n)}} and~\ref{lemma0}.

\subsection{Treewidth of \texorpdfstring{$H(t,q,n)$}{H(t,q,n)}}

In this subsection we analyze the asymptotic behavior of $tw(H(t,q,n))$ when $n$ goes to infinity and give the proof of Theorem~\ref{thm: asymptotic tw(H(t,q,n))}.



We first prove the lower bound of $tw(H(t,q,n))$ by Proposition~\ref{prop: lowerbound treewidth counting neighbors}.

\vspace{.2cm}
\begin{proposition}[Lemma 7 of~\cite{tw_pw_HammingGraph_chandran2006treewidth}]\label{prop: lowerbound treewidth counting neighbors}
	Let $G(V,E)$ be a graph with $n$ vertices. If for each subset $X$ of $V$ with $n/4 \le |X| \le n/2$, $\varPhi(X) \ge k$, then $tw(G) \ge k-1$.
\end{proposition}

\vspace{.2cm}
\begin{lemma}\label{lemma: lowerbound of tw(H(t,q,n))}
	$tw(H(t,q,n)) \ge c_1 t q^{n}/\sqrt{n}$ for some constant $c_1$ not depending on $t$ or $q$ when $n$ is sufficiently large.
\end{lemma}

\begin{proof}
By Proposition~\ref{prop: lowerbound treewidth counting neighbors}, we have $tw(H(t,q,n)) \ge \min b_v(m, H(t,q,n)) - 1$ over integers $m$ in the range $q^n/4 \le m \le q^n/2$.
So it is sufficient to give a lower bound for $b_v(m, H(t,q,n))$ over $ m \in [q^n/4, q^n/2]$.

In~\cite{bandwidth_q2inf_harper1999isoperimetric}, Harper showed that\footnote{This statement is not explicitly stated in~\cite{bandwidth_q2inf_harper1999isoperimetric}, but can be easily inferred from Theorem~3 on pp.~302.} if
\begin{equation}
	\label{eq: continuous lower bound of counting neighbors}
	m = q^n \sum_{i=0}^{r} \binom{n}{i} x^{n-i}{(1-x)}^{i} \mbox{ for some $x, r$, $0 < x < 1$,}
\end{equation}

then
\begin{equation}
	b_v(m, H(t,q,n)) \ge q^n \min_{x, r} \left\{ \sum_{i=1}^{t} \left( \binom{n}{r+i} x^{n-r-i}{(1-x)}^{r+i} \right) \right\},
\end{equation}

where the minimum is taken over all $x,r$ satisfying Eq~\ref{eq: continuous lower bound of counting neighbors}.
	





In~\cite{tw_pw_HammingGraph_chandran2006treewidth}, it is proved that when $m = q^n \sum_{i=0}^{r} \binom{n}{i} x^{n-i}{(1-x)}^{i}$, we have

\begin{equation*}
	n(1-x) - \sqrt{4nx(1-x)} < r < n(1-x) + \sqrt{4nx(1-x)}.
\end{equation*}

By Stirling's approximation, it can be shown that for all $r$ in the above range, we have
\[
	q^n \binom{n}{r+i}x^{n-r-i}{(1-x)}^{r+i} \ge c_1(q^n/\sqrt{n}), 1 \le i \le t
\]
for some constant $c_1 > 0$ not depending on $t$ and $q$.
Then, $b_v(m, H(t,q,n)) \ge c_1tq^n/\sqrt{n}$.
\end{proof}

Then we intend to estimate the upper bound via bandwidth.

\vspace{.2cm}
\begin{lemma}\label{lemma: upper bound of pw(H(t,q,n))}
	$tw(H(t,q,n)) \le pw(H(t,q,n)) \le c_2 t q^n/\sqrt{n}$ for some constant $c_2$ not depending on $t$ or $q$ when $n$ is sufficiently large.
\end{lemma}

\begin{proof}
For convenience, we first assume $q$ is even.
The case when $q$ is odd can be handled similarly.
Let $f$ be a function from $[q]$ to $\{0,1\}$ as follows:
\begin{equation*}
	f(i) = \left\{ \begin{array}{cc}
		0 & \mbox{if $1 \le i \le q/2$,} \\
		1 & \mbox{if $q/2 < i \le q$.} \\
\end{array} \right.
\end{equation*}

Suppose that $(a_1, a_2, \ldots, a_n) \in {[q]}^{n}$ is an $n$-vector corresponding to a vertex $x$ of $H(t,q,n)$.
Define function $g$ from  $V(H(t,q,n))$ to  $V(H(t,2,n))$ that maps $x \in V(H(t,q,n))$ to the vertex $g(x) \in V(H(t,2,n))$ which corresponds to the vector $(f(a_1), f(a_1), \ldots, f(a_n))$.
Note that $g$ maps exactly ${(q/2)}^n$ vertices of $H(t,q,n)$ to a given vertex of $H(t,2,n)$.

Let $H(t,2,n)$ have a path decomposition  whose bags are $\{P_i\}$.
By replacing each $P_i$ with $P_i' = \bigcup_{y \in P_i} g^{-1}(y)$,  it is easy to show that $\{P_i'\}$ is a path decomposition of $H(t,q,n)$.
Therefore, $pw(H(t,q,n)) \le pw(H(2,q,n)) \cdot {\left(\frac{q}{2}\right)}^n$.

From Theorem~\ref{thm: bandwidth of generalized Hypercube} and Stirling's approximation, we have $bw(H(t,2,n)) \le c_2t2^n/\sqrt{n}$ for some constant $c_2$.
Therefore, we have $tw(H(t,q,n)) \le pw(H(t,q,n)) \le c_2 t q^n/\sqrt{n}$.
\end{proof}

Combining Lemmas~\ref{lemma: lowerbound of tw(H(t,q,n))} and~\ref{lemma: upper bound of pw(H(t,q,n))}, we can derive Theorem~\ref{thm: asymptotic tw(H(t,q,n))}.

\section{Treewidth of bipartite Kneser graph and Johnson graph}\label{section: treewidth of bipartite Kneser and Johnson}

For positive integers $n$ and $k$ satisfying $n \ge 2k+1$, the bipartite Kneser graph $BK(n, k)$ has all subsets of $[n]$ with $k$ or $n-k$ elements as vertices and an edge between any two vertices when one is a subset of the other.
It is also called \textit{middle cube graph}.
In the following, we will use $k$-subsets and $(n-k)$-subsets of $[n]$ to represent the vertices of bipartite Kneser graph.
We call vertices that are $k$-subsets of $[n]$ the \textit{left part} denoted by $V_L$, and the rest is called the \textit{right part} denoted by $V_R$.
$V_L$ and $V_R$ are two parts of bipartite Kneser graph $BK(n,k)$.

For positive integers $n$ and $k$ satisfying $n > k$, the \textit{Johnson graph} $J(n,k)$ has all subsets of $[n]$ with $k$ elements and an edge between any two vertices when their intersection has exactly $(k-1)$ elements.
Since we have a bijection between all $k$-subsets of $[n]$ and all binary $n$-vectors with exactly $k$ ones,
we can also treat a vertex of $J(n,k)$ as an $n$-vector with exactly $k$ ones.
From this point of view, two vertices are adjacent iff their the Hamming distance of their corresponding $n$-vectors is no more than $2$.
Therefore, $J(n, k)$ is the $k$-th slice of $H(2,2,n)$, that is, $J(n,k)$ is a subgraph of $H(2,2,n)$ induced by vertices from $H(2,2,n)$ corresponding to $n$-vector with exactly $k$ ones.

Let $G(V,E)$ be a graph and $S \subseteq V$  a vertex subset of $G$. Denote the subgraph of $G$ induced by $S$ as $G[S]$.
Recalling the definition of $V_k^{(t,n)}$ in subsection 3.1, we have Proposition~\ref{prop: J(n,k) slice}.

\begin{proposition}\label{prop: J(n,k) slice} We have
	$J(n,k) \cong H(2,2,n)[V^{(2,n)}_k] $.
\end{proposition}

\subsection{Treewidth of \texorpdfstring{$BK(n, k)$}{BK(n, k)} when \texorpdfstring{$n$}{n} is large enough}

In this section, we focus on the treewidth of $BK(n,k)$ when $n$ is large enough and give the proof of Theorem~\ref{thm: tw of BK(n,k) when n is large enough}.
Before proof, we need the following proposition.

\begin{proposition}[\cite{tw_Kneser_Harvey2014}]\label{prop: upperbound of treewidth by independence number}
	For any graph $G$, $tw(G) \le \max\{\Delta(G), |V(G)| - \alpha(G) - 1  \}$, where $\alpha(G)$ is the independent number of $G$.
\end{proposition}

\begin{lemma}\label{lemma: upperbound treewidth of BK(n,k) when n is large enough} We have
	$tw(BK(n,k)) \le \binom{n}{k} - 1$.
\end{lemma}

\begin{proof}
From~\cite{factorization_of_BipartiteKneser_jin20201}, we have that bipartite Kneser graph has a perfect matching.
Consequently, $\alpha(BK(n,k)) = \binom{n}{k}$.
Note that $BK(n,k)$ is a regular graph with order $2\binom{n}{k}$ and $\Delta(BK(n,k)) = \binom{n-k}{k}$.
By Proposition~\ref{prop: upperbound of treewidth by independence number}, we have $tw(BK(n,k)) \le \binom{n}{k} - 1$.
\end{proof}

\vspace{.2cm}
\begin{lemma}\label{lemma: lowerbound treewidth of BK(n,k) when n is large enough}
	When $k \ge 2$ and $3\binom{n-k}{k} \ge 2\binom{n}{k}$, $tw(BK(n,k)) \ge \binom{n}{k} - 1$.
\end{lemma}

\begin{proof}
Denote $BK(n,k)$ by $G$ and $V(G)=V_L\cup V_R$, where $V_L=\binom{[n]}{k}$ and $V_R=\binom{[n]}{n-k}$.
Suppose  $tw(G) < \binom{n}{k} - 1$. From Corollary~\ref{coro: tw and separator 2/3}, there exists a separator $X$ of $G$ with $|X| < \binom{n}{k}$ such that
there exists non-empty vertex set $A$ and $B$ with $A \cup B = V(G) - X$, $A \cap B = \emptyset$, $ |V(G)-X|/3 \le |A|,|B| \le 2|V(G)-X|/3$ and there is no edge between $A$ and $B$.
Let  $A_L = A \cap V_L$, $A_R = A \cap V_R$, $B_L= B \cap V_L$ and $B_R = B \cap V_R$~(see Figure~\ref{fig: BK}).

\begin{figure}[t]
	\centering         
	\includegraphics[width=0.6\linewidth]{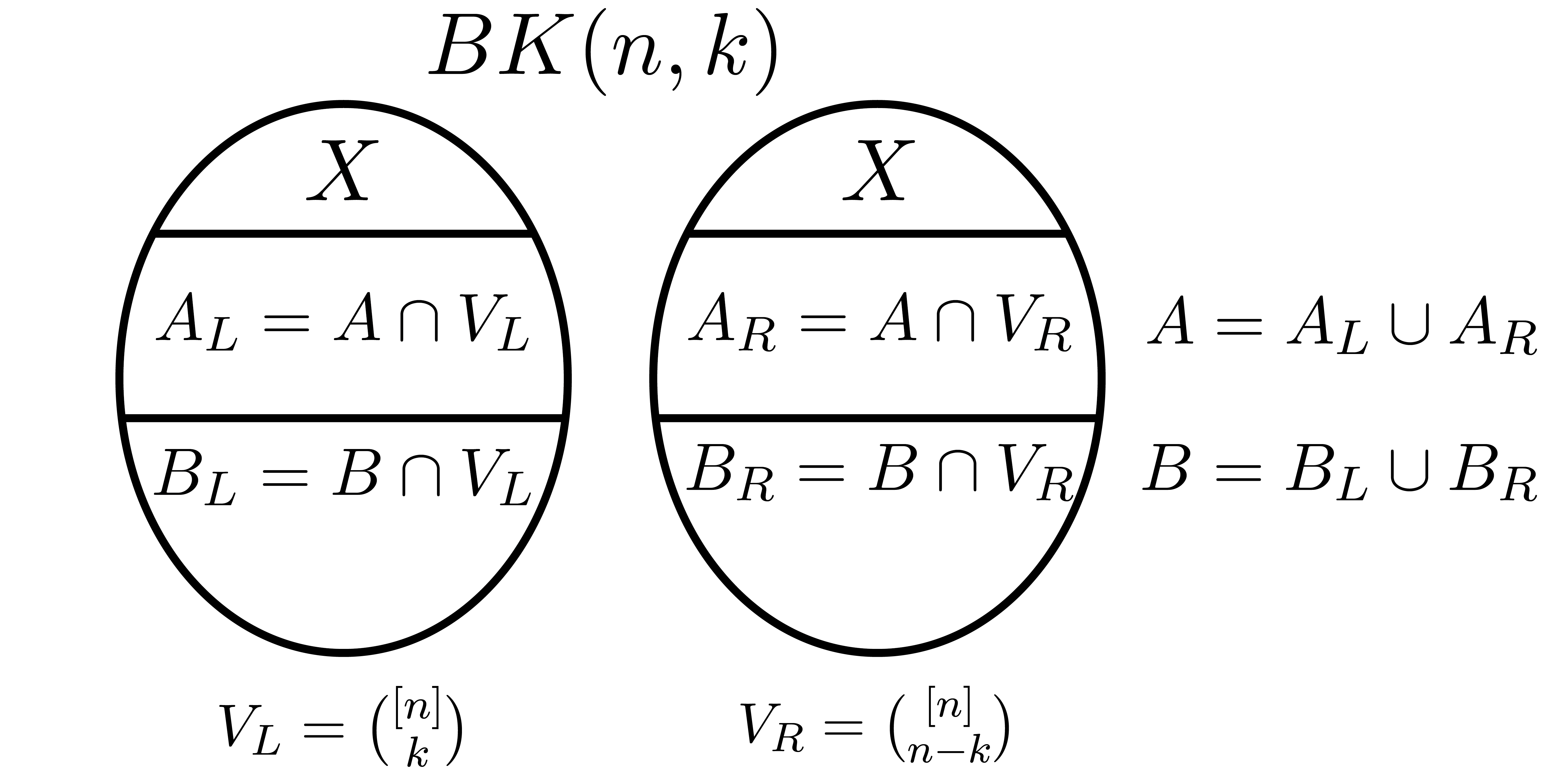}
	\caption{Sketch graph of $BK(n,k)$. The condition that there is no edge between $A$ and $B$ equals to there is no edge between $A_L$ and $B_R$, and between $A_R$ and $B_R$.}\label{fig: BK}
\end{figure}

Since $A,B$ are nonempty and $|X| < \binom{n}{k}$,  we assume without loss of generality that $A_L$ and $B_R$ are nonempty.
Let
\begin{equation*}
\begin{aligned}
	\mathcal{A} &= \left\{ S \in \binom{[n]}{k} \mid S \in A_L \right\}, \\
	\mathcal{B} &= \left\{ S \in \binom{[n]}{n-k} \mid S \in B_R \right\}.
\end{aligned}
\end{equation*}
Let $\mathcal{C} = \{ [n] - S \mid S \in \mathcal{B}\} \subseteq \binom{[n]}{k}$. Since there is no edge between $A_L$ and $B_R$, for any $S_1 \in \mathcal{A}$ and $S_2 \in \mathcal{B}$, $S_1 \nsubseteq S_2$
which implies $S_1 \cap ([n] - S_2) \neq \emptyset$. 

\begin{definition}[Cross-intersecting families]
	Let $\mathcal{A}$ and $\mathcal{B}$ be two families of subsets of a finite set $X$.
	We say that $\mathcal{A}$ and $\mathcal{B}$ are \textit{cross-intersecting} if for any $A \in \mathcal{A}$ and $B \in \mathcal{B}$, $A \cap B \neq \emptyset$.
\end{definition}

Hence, $\mathcal{A}$ and $\mathcal{C}$ is cross-intersecting.
By the properties of cross-intersecting families from~\cite{crossintersecting_frankl1992some}, we have $|\mathcal{A}|+|\mathcal{C}| \le \binom{n}{k} - \binom{n-k}{k} +1$.
Notice that $|\mathcal{A}| = |A_L|$ and $|\mathcal{C}| = |\mathcal{B}| = |B_R|$.
Then we have $|A_L|+|B_R| \le \binom{n}{k} - \binom{n-k}{k} +1$.

\textbf{If $A_R, B_L$ are both nonempty,} then, similarly, we have $|A_R| + |B_L| \le \binom{n}{k} - \binom{n-k}{k} +1$.
Hence, $|A|+|B| \le 2\left( \binom{n}{k} - \binom{n-k}{k} +1 \right)$ and $|X| = |V(G)| - |A| - |B| \ge 2\binom{n-k}{k} - 2$.
Since $k \ge 2$ and $3\binom{n-k}{k} \ge 2 \binom{n}{k}$, we have $2\binom{n-k}{k} - 2 \ge \binom{n}{k}$ which derives a contradiction with $|X| < \binom{n}{k}$.

\textbf{If $A_R, B_L$ are both empty}, then we have $|A| + |B| = |A_L| + |B_R| \le \binom{n}{k} - \binom{n-k}{k} +1$.
Hence, $|X| = |V(G)| - |A| - |B| \ge \binom{n}{k} + \binom{n-k}{k} - 1 \ge \binom{n}{k}$, a contradiction.

\textbf{If there is only one empty set in $\{A_R, B_L\}$}, say $A_R \neq \emptyset$ and $B_L=\emptyset$, then
we have $|B| = |B_R| \le \binom{n}{k} - \binom{n-k}{k}$ by $|A_L|+|B_R| \le \binom{n}{k} - \binom{n-k}{k} +1$ and $A_L\not=\emptyset$. Since $|V(G)-X|/3 \le |A|,|B| \le 2|V(G)-X|/3$, we have $|A| + |B| \le 3|B| \le 3\left( \binom{n}{k} - \binom{n-k}{k} \right)$.
Since $3\binom{n-k}{k} \ge 2 \binom{n}{k}$, we have $|X| = |V(G)| - |A| - |B| \ge 3\binom{n-k}{k} - \binom{n}{k} \ge \binom{n}{k}$, a contradiction.
\end{proof}

Theorem~\ref{thm: tw of BK(n,k) when n is large enough} can be easily derived from Lemmas~\ref{lemma: upperbound treewidth of BK(n,k) when n is large enough} and~\ref{lemma: lowerbound treewidth of BK(n,k) when n is large enough}.

\subsection{Treewidth of \texorpdfstring{$BK(2k+1, k)$}{BK(2k+1, k)} and \texorpdfstring{$J(2k+1, k)$}{J(2k+1, k)}}

When $n$ is large enough, the treewidth of the bipartite Kneser graph can be exactly calculated by Theorem~\ref{thm: tw of BK(n,k) when n is large enough}.
Now we focus on the treewidth of the bipartite Kneser graph when $n$ is small and give the proof of Theorem~\ref{thm: tw BK J uplow}.


In order to prove our reslt, we need more definitions. A graph $G$ is chordal if and only if, in any cycle of
length larger than 3 in $G$, there exists a chord connecting two nonadjacent vertices of the cycle. Given a graph $G$, define $\omega(G)$, the clique number  of $G$ to be the number of vertices of the largest clique in $G$.
The treewidth of a graph $G$ has a close relationship with its chordal supergraph as Proposition~\ref{prop: treewidth and chordal relation} shows.

\vskip.2cm
\begin{proposition}[\cite{GraphMinor_2_robertson1986graph}]\label{prop: treewidth and chordal relation}
Given a graph $G$, $tw(G) = \min \{ \omega(H) - 1\mid G \subseteq H, \mbox{$H$ is chordal}\}$.
\end{proposition}

Using Proposition~\ref{prop: treewidth and chordal relation}, we can prove the relationship between the treewidth of the bipartite Kneser graph and Johnson graph as Lemma~\ref{lemma: treewidth relationship between bipartite Kneser graph and Johnson graph}. 
We first explain the proof ideas. Note that $BK(n,k)$ is a bipartite graph with two parts $\binom{[n]}{k}$ and $\binom{[n]}{n-k}$. And the vertex set of $J(n,k)$ is $\binom{[n]}{k}$ which is the same as one part of $BK(n,k)$. Thus, we can embed $J(n,k)$ into $BK(n,k)$ by letting the $\binom{[n]}{k}$ part of $BK(n,k)$ is isomorphic to $J(n,k)$. We then will show the result graph is chordal.
After calculating its clique number and then using Proposition~\ref{prop: treewidth and chordal relation}, we can derive our result.

\begin{lemma}\label{lemma: treewidth relationship between bipartite Kneser graph and Johnson graph}
	For any positive integer $k $, $tw(BK(2k+1,k)) \le tw(J(2k+1,k)) $.
\end{lemma}

\begin{proof}
Let $n= 2k+1$.
By Proposition~\ref{prop: treewidth and chordal relation}, there exists a chordal graph $H$ such that $J(n,k)$ is a subgraph of $H$ and $\omega(H) - 1 = tw(J(n,k))$.
Add edges to the left part of $BK(n,k)$ such that the left part is isomorphic to $H$ and denote the result graph as $BK'(n,k)$.
That is, $BK'(n,k)[V_L] \cong H$.

We first claim that $BK'(n,k)$ is chordal. Suppose $C$ is a cycle in $BK'(n,k)$ with length larger than $3$.
We first consider the case that there is a vertex $v \in V(C)\cap V_R$. Let $u_1$ and $u_2$ be the two neighbors of $v$ on $C$.
Then $u_1,u_2\in V_L$. Let $A_v \in \binom{[n]}{k+1}$ be the corresponding set of $v$ and $A_{u_1}, A_{u_2} \in \binom{[n]}{k}$ the corresponding sets of $u_1$ and $u_2$ respectively.
Since $u_1v, u_2v \in E(BK'(n,k))$, we have $A_{u_1} \subseteq A_v$ and $A_{u_2} \subseteq A_v$.
Then $|A_{u_1} \cap A_{u_2}| = k-1$. Hence, $u_1u_2 \in E(BK'(n,k))$.
Now assume all vertices of $C$ are from the left side. Since $H$ is chordal, there must exist a chord in $C$.
From above all, we have that $BK'(n,k)$ is chordal.

Let $W$ be a clique of $BK'(n,k)$.
If there exists a vertex $v \in W$ from the right side, then $|W| \le k+1+1$ by $d(v) = k+1$.
If all vertices of $W$ are from the left part, then $|W| \le \omega(H) = tw(J(n,k)) + 1$.
Hence
\begin{equation*}
 tw(BK(n,k)) \le \omega(BK'(n,k))-1 \le \max\{tw(J(n,k)), k+1\}.
\end{equation*}

Notice that $J(2k+1, k)$ is $k(k+1)$-regular. Since $tw(G) \ge \delta(G)$~\cite{tw_lb_degree_koster2005degree} for any graph $G$, we have $tw(J(2k+1,k)) \ge k(k+1) \ge k+1$.
Then $tw(BK(2k+1, k)) \le tw(J(2k+1, k))$.
\end{proof}

\begin{lemma}\label{lemma: upper bound treewidth of Johnson graph}
	For positive integers $n$ and $k$ satisfying $n>k$,
	$tw(J(n,k)) \le bw(M^{(2,n)}_{k,k}) = r(M^{(2,n)}_{k,k}) - \binom{n}{k}$.
\end{lemma}

\begin{proof}
Let $G=H(2,2,n)[V_k^{(2,n)}]$ for short.
By Proposition~\ref{prop: J(n,k) slice}, we have $G \cong J(n,k)$.

Let $t=2$ in subsection 3.1. Recalling that $\eta' \triangleq \eta_{k}^{(n)}$ is an ordering of $V_k^{(2,n)}$, then $\eta'$ is also an ordering of $G$.
Note that the adjacency matrix of $G$ with the ordering $\eta'$ is exactly $M_{k,k}^{(2,n)}$.
By Remark~\ref{rmk: bandwidth} and the definition of bandwidth, we have
\begin{equation}
	bw(M_{k,k}^{(2,n)}) = \max \limits_{e \in E(G)}\Delta(e, \eta') \ge \min_{\eta} \max \limits_{e \in E(G)}\Delta(e, \eta) = bw(G) = bw(J(n,k)).
\end{equation}

Combining Eq~\ref{eq: tw pw bw inequality} and Eq~\ref{eq: radius bandwidth relation} with $s=|V_k^{(2,n)}|=\binom{n}{k}$, we can derive the lemma.
\end{proof}

Specifically, take $n=2k+1$, and then from Lemma~\ref{lemma: guessing r^(t,n)_k1,k2},
we can derive the asymptotic behavior of $bw(M^{(2,2k+1)}_{k,k})$ by calculating:

\begin{equation}
\begin{aligned}
	\lim \limits_{k \rightarrow +\infty} bw(M^{(2,2k+1)}_{k,k})/\binom{2k+1}{k} & = 1/2.
\end{aligned}
\end{equation}

Therefore, $tw(J(2k+1,k)) = O(\binom{2k+1}{k})$.

\vspace{.2cm}
\begin{proposition}[\cite{tw_lb_spectral_chandran2003spectral}]\label{prop: tw lb spectral}
	Suppose $G$ is a $k$-regular graph with $n$ vertices and $A(G)$ is the adjacency matrix of $G$.
	Let $\mu(G) = k - \lambda(G)$ where $\lambda(G)$ is the second-largest eigenvalue of $A(G)$.
	Then
	\[
	tw(G) \ge \left\lfloor \frac{3n}{4} \frac{\mu(G)}{\Delta(G) + 2\mu(G)} \right\rfloor - 1.	
	\]
\end{proposition}

\vspace{.2cm}
\begin{proposition}[\cite{spectral_middle_cube_qiu2009spectrum}]\label{prop: spectral middle cube}
	The characteristic polynomial of $BK(2k+1,k)$ is
	\[
		\prod_{i=1}^{k+1}{(\lambda + i)}^{\binom{n}{k+1-i} - \binom{n}{k-i}}	{(\lambda - i)}^{\binom{n}{k+1-i} - \binom{n}{k-i}}.	
	\]
\end{proposition}

\vspace{.2cm}
\begin{lemma}\label{lemma: tw estimation BK(2k+1,k)}We have
	\[
	tw(BK(2k+1,k)) \ge 	\left\lfloor \frac{3}{2} \binom{2k+1}{k} \frac{1}{k+3} \right\rfloor - 1.	
	\]
\end{lemma}

Lemma~\ref{lemma: tw estimation BK(2k+1,k)} can be derived from Propositions~\ref{prop: tw lb spectral} and~\ref{prop: spectral middle cube}.
Theorem~\ref{thm: tw BK J uplow} can be derived by Lemmas~\ref{lemma: treewidth relationship between bipartite Kneser graph and Johnson graph},~\ref{lemma: upper bound treewidth of Johnson graph} and~\ref{lemma: tw estimation BK(2k+1,k)}.

\section{Treewidth of generalized Petersen graph}\label{section: tw generalized Petersen graph}

In this section, we determine the treewidth of generalized Petersen graph.
The vertex set and edge set of generalized Petersen graph $G_{n,k}$ are
\begin{equation*}	
\begin{aligned}
	V(G_{n,k})&=\{v_1,\hdots,v_n,u_1,\hdots,u_n\},\\ E(G_{n,k})&={\{v_i u_i\}} \cup {\{v_i v_{i+1}\}} \cup {\{u_i u_{i+k}\}}, i=1,2,\hdots,n,
\end{aligned}
\end{equation*}
where subscripts are to be read modulo $n$ and $k < n/2$.
Let $G$ be a graph and $X$ and $Y$ are two connected subgraphs of $G$.
We say $X$ \textit{touches} $Y$ when $V(X) \cap V(Y) \neq \emptyset$ or there exists an edge between $X$ and $Y$.
A \textit{bramble} of  $G$ is a family of connected subgraphs of $G$ that all touch each other.
Let $S$ be a subset of $V(G)$. $S$ is said to be a \textit{hitting set} of bramble $B$ if $S$ has nonempty intersection with each of the subgraphs in $B$.
The order of a bramble is the smallest size of a hitting set.
Brambles may be used to characterize the treewidth of a given graph.

\vspace{.2cm}
\begin{proposition}[\cite{bramble_seymour1993graph}]\label{prop: bramble}
	Let $G$ be a graph. Then $tw(G) \ge k$ if and only if $G$ contains a bramble of order at least $k+1$.
\end{proposition}

With the help of Proposition~\ref{prop: bramble}, we now can give the proof of Theorem~\ref{thm: tw generalized Petersen}.

\noindent
\textit{Proof of Theorem~\ref{thm: tw generalized Petersen}: }
	First, we intend to prove the lower bound.
	Construct a bramble $B = \{B_i\}$ of $G_{n,k}$ as
\begin{equation*}
\begin{aligned}
	V_i&=\{v_i,v_{i+1},\hdots,v_{i+t},u_{i+t},u_{i+t+k},u_{i+t+2k},\hdots,u_{i+t+tk}\}, \\
	B_i&=G_{n,k}[V_i],
\end{aligned}
\end{equation*}
where $t=\lceil \frac{n}{2k+2}\rceil, i=1,2,\hdots,n$.
Then we have $|V_i|=2t+2$ and $B_i$ is connected.
For each pair $i$ and $j$, we intend to prove that $B_i$ touches $B_j$.
Without loss of generality, assume $1\leq i<j \leq n$.

\begin{itemize}
	\item If $j\leq i+t$, then $v_j\in V_i\cap V_j$.
	\item If $i+t<j\leq i+t+tk$, let $r$ be the minimum integer in $\{0,1,\hdots,k-1\}$ such that $j+r\equiv i+t$~(mod~$k$). Noticing that $t \ge k$ when $n \geq 8{(2k+2)}^{2}$. Hence, $v_{j+r}\in V_j, u_{j+r}\in V_i$, and $u_{j+r}v_{j+r}\in E(G_{n,k})$.
	\item If $i+t+tk<j<n+i-t$, then $j+t<i+n<j+n-t-tk\leq j+t+tk$. The last inequality comes from  $2kt+2t\geq n$ by $t=\lceil \frac{n}{2k+2}\rceil$ and then we have $t+tk\geq n-t-tk$ which implies $j+t<i+n\leq j+t+tk$. The next proof is the same as that in the second situation.
	\item If $n+i-t\leq j\leq n$, the proof is the same as that in the first situation.
\end{itemize}

Then $B$ is a bramble.
Let $S$ be a hitting set of $B$. We construct a hypergraph $H$ with vertex set $V(H)=V(G_{n,k})$ and hyperedge set ${\{V_i\}}_{1 \le i \le n}$.

\begin{definition}[Transversal]
	Let $H$ be a hypergraph on a set $X$ with edges $E_1,\ldots,E_m$. 
	A set $T \subseteq X$ is called a transversal of $H$ if $T$ intersects every edge of $H$, that is,
	\[
	T \cap E_i \neq \emptyset, \forall i=1,2,\ldots,m.
	\]
\end{definition}

Then $S$ is a transversal of $H$ and thus, $\min |S| = \tau(H)$ where $\tau(H)$ is the transversal number of $H$.
Since $\tau(H) \ge  \max \limits_{H' \subseteq H} \frac{m(H')}{\Delta(H')}$~\cite{berge1984hypergraphs}, where $m(H')$ is the number of edges in $H'$ and $\Delta(H')$ is the maximum degree of $H'$, we have that the order of $B$ is
\[
	\min |S| \ge \frac{m(H)}{\Delta(H)} = \frac{n}{t+1}.
\]


Since  $t=\lceil \frac{n}{2k+2}\rceil$, we have $t-1\leq \frac{n}{2k+2}$ and
\[
	2k+2-\frac{n}{t+1}\leq \frac{n}{t-1}-\frac{n}{t+1}=\frac{2n}{t^2-1}< 1.
\]

Therefore, the order of $B$ is no less than $2k+2$.
From Proposition~\ref{prop: bramble}, we can derive that $tw(G_{n,k})\geq 2k+1$.

The upper bound can be proved via construction.
Construct a path-decomposition of $G_{n,k}$ as following~(see Figure~\ref{fig: Petersen 1} and Figure~\ref{fig: Petersen 2}).

\begin{figure}[t]
	\centering         
	\includegraphics[width=\linewidth]{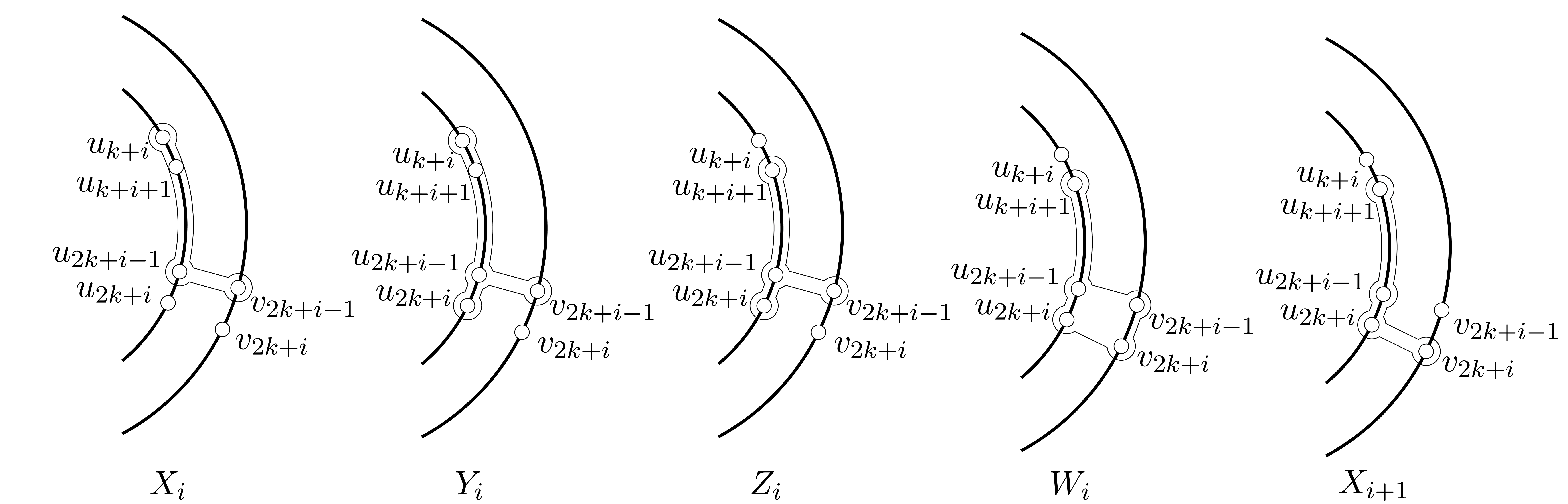}
	\caption{The vertex sets $X_i, Y_i, Z_i$ and $W_i$.}\label{fig: Petersen 1}
\end{figure}

\begin{figure}[t]
	\centering         
	\includegraphics[width=0.5\linewidth]{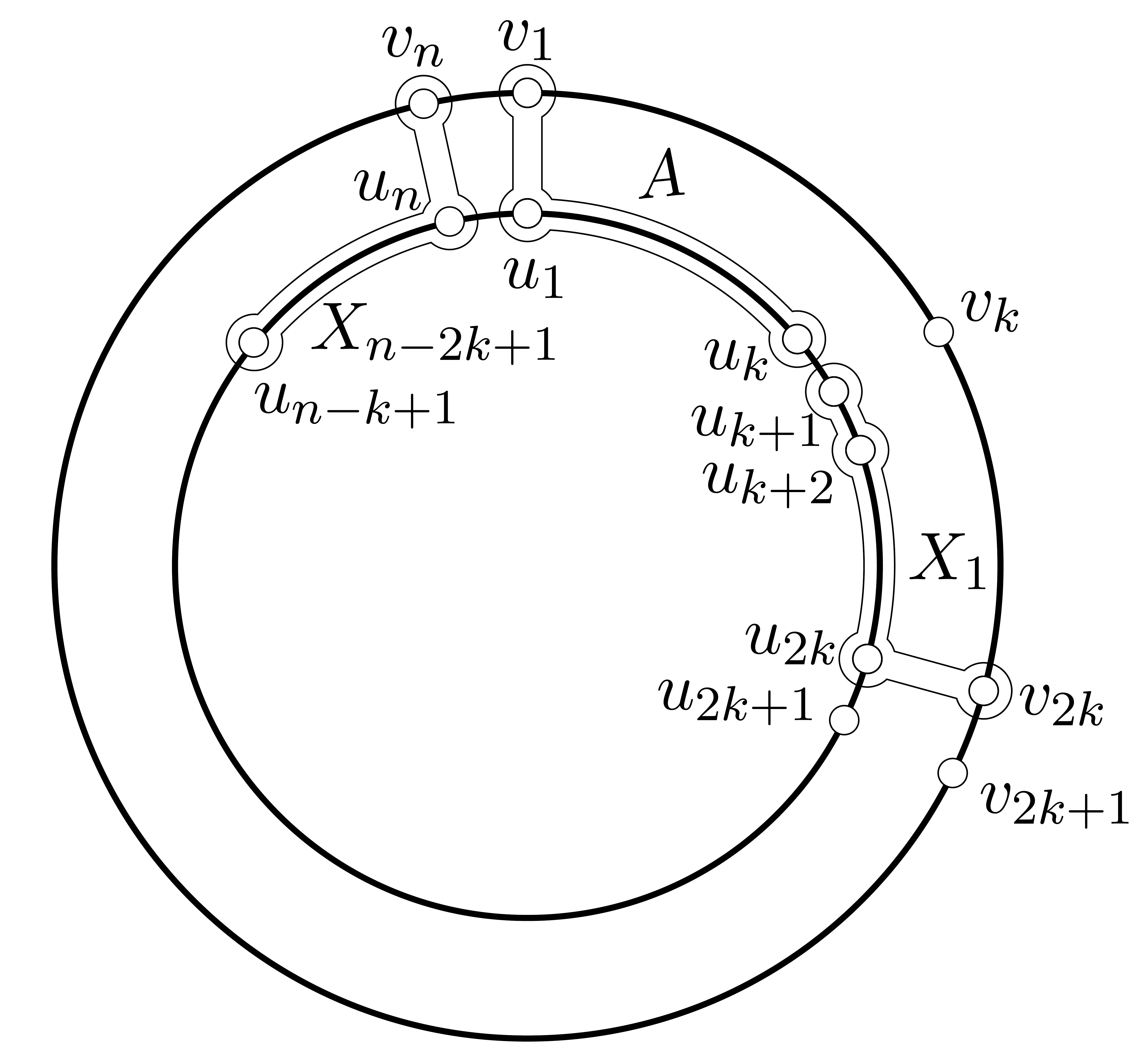}
	\caption{The defined sets on generalized Petersen graph.}\label{fig: Petersen 2}
\end{figure}

\begin{enumerate}[Step 1.]
	\item Let $A = \{v_1, u_1, \ldots, u_k\}$, $B_1=\{v_1, \ldots, v_{k}\}$ and $B_2=\{v_{k}, \ldots, v_{2k}\}$.
	\item Let $X_1 = \{u_{k+1}, \ldots, u_{2k}, v_{2k}\}$.
	\item Recursively define $Y_i = X_i \cup \{u_{2k+i}\}$, $Z_i = Y_i - \{u_{k+i}\} $, $W_i = Z_i \cup \{v_{2k+i}\}$ and $X_{i+1} = W_i - \{v_{2k+i-1}\}$ for $1 \le i \le n-2k$.
	It is easy to verify that $X_i = \{ u_{k+i}, \ldots, u_{2k+i-1}, v_{2k+i-1}\}$.
	\item Define $X_{n-2k+1} = W_{n-2k} - \{v_{n-1}\}$.
	\item Define a path decomposition $\mathcal{P}_1$ of $G_{n,k} - A$ by successively connect
	\[
		(B_1,B_2, X_1, Y_1, Z_1, W_1, X_2, \ldots, W_{n-2k}, X_{n-2k+1}).
	\]
	\item Add each vertex in $A$ to all bags of $\mathcal{P}_1$, then we obtain a path decomposition $\mathcal{P}_2$ of $G_{n,k}$.
\end{enumerate}


Here we explain why we construct the path-decomposition in this way. For a cycle $v_{1}v_{2}\ldots v_{k}$, the way to build its path-decomposition is to first delete any vertex, say $v_1$, then the left part is a path $v_2\ldots v_k$. Then build a path-decomposition like $\{v_2v_3\}-\{v_3v_4\}-\ldots-\{v_{k-1}v_k\}$ and then add $v_1$ to all bags. Here the generalized Petersen graph behaves like a ``double-cycle''. We first delete $A$, and the rest part behaves like a ``path''. Then the sequence of $X_i,Y_i,Z_i,W_i,X_{i+1}$ is like what we do in the path-decomposition of a cycle. Once we move one vertex so that the width do not increase too much. $B_1$ and $B_2$ are designed to cover the left vertices. Finally we add $A$ to all bags just like for the cycle we add $v_1$ back.

It is easy to verify $\mathcal{P}_1$ is a path-decomposition of $G_{n,k}-A$ by checking the three properties which implying that $\mathcal{P}_2$ is a path-decomposition of $G_{n,k}$.

Since $|X_i| = |Z_i| = k+1, |Y_i| = |W_i| = k+2, |B_1| = k, |B_2| = k + 1$ and $|A| = k+1$, the width of $\mathcal{P}_2$ is $2k+3$ and, hence, $tw(G_{n,k}) \le pw(G_{n,k}) \le 2k+2$.
\hfill$\square$

\section*{Acknowledgement}
The authors would like to thank the referee for their careful reading and valuable comments which help to improve the presentation of this paper.
M. Cao is supported by the National Natural Science Foundation of China (Grant 12301431), M. Lu is supported by the National Natural Science Foundation of China (Grant 12171272).

\bibliography{ref.bib}
\bibliographystyle{wyc3}

\section*{Appendix A:\@ Proof of Lemma~\ref{lemma: guessing r^(t,n)_k1,k2}}

Before the proof, we need several definitions. We say the parameter tuple $(t,n,k,s)$ is \textit{valid} if $t \ge 2s, n \ge 1, 1 \le t < n - 1$ and $k+t-2s\le n$, that correspons to the non-trivial situation of Lemma~\ref{lemma: guessing r^(t,n)_k1,k2}.
The parameter tuple of matrix in the formula of $r^{(i)}_{t,n,k,t-2s}~( i=1,2,3,4)$ is shown as Table~\ref{tab: parameter tuple}.

\begin{table}[htbp]
	\renewcommand{\arraystretch}{1.5}
	\centering
	\begin{tabular}{cccccc}
		\hline 
		$r^{(i)}$-term	& matrix & $t'$ & $n'$ & $k'$ & $s'$ \\ 
		\hline 
		$r^{(1)}_{t,n,k,t-2s}$& $r(M^{(t,n-1)}_{k-1,k+t-2s-1})$ & $t$ & $n-1$ & $k-1$ & $s$ \\
		$r^{(2)}_{t,n,k,t-2s}$& $r(M^{(t-1,n-1)}_{k-1,k+t-2s})$ & $t-1$ & $n-1$ & $k-1$ & $s-1$ \\
		$r^{(3)}_{t,n,k,t-2s}$& $r(M^{(t-1,n-1)}_{k,k+t-2s})$ & $t-1$ & $n-1$ & $k $& $s$ \\
		$r^{(4)}_{t,n,k,t-2s}$& $r(M^{(t,n-1)}_{k,k+t-2s})$ & $t$ & $n-1$ & $k$ & $s$ \\
		\hline 
	\end{tabular}
	\caption{Parameter tuple of the matrix in the Eq~\ref{eq: lemma: guessing r^(t,n)_k1,k2} of $r^{(i)}_{t,n,k,t-2s}$. }\label{tab: parameter tuple}
\end{table}

First, it is easy to verify that $D_{t,n,k,s} = A^{(1)}_{t,n,k,s} + A^{(2)}_{t,n,k,s} + C_{t,n,k,s}$ holds when $k-s=0$ and, $D_{t,n,k,s} = B_{t,n,k,s} + C_{t,n,k,s}$ when $k-s=n-t$ and $A^{(1)}_{t,n,k,s} + A^{(2)}_{t,n,k,s} = B_{t,n,k,s}$ hold when $ k-s = \lfloor (n-t)/2 \rfloor$.
Hence, in those situations, we can calculate $r(M^{(t,n)}_{k,k+t-2s})$ in both ways.

In the following, we prove the Lemma by induction on $t+n+k+s$. From Table~\ref{tab: parameter tuple}, when we calculate $\{r^{(i)}_{t,n,k,t-2s}\}_{1 \le i \le 4}$, the term always obtains a smaller $t'+n'+k'+s'$, which means we can use Eq~\ref{eq: lemma: guessing r^(t,n)_k1,k2} to calculate those terms by induction.

\vspace{1em}
\noindent
\textbf{Step 1: verify Lemma~\ref{lemma: guessing r^(t,n)_k1,k2} when either $t \ge n-1$, or $t=1$.}
\vspace{1em}

When $t \ge n-1$, it is trivial to verify $M^{(t,n)}_{k,k+t-2s}$ is an all-one matrix and, thus, $r(M^{(t,n)}_{k,k+t-2s}) = \binom{n}{k} + \binom{n}{k+t-2s} -1 $ from definition.

When $t=1$, then $s=0$.
In this case $r^{(2)}_{1,n,k,1} = -\infty$ and $r^{(3)}_{1,n,k,1} = -\infty$ and,
hence, $r(M_{k,k+1}^{(1,n)}) = \max \{ r^{(1)}_{1,n,k,1}, r^{(4)}_{1,n,k,1} \}$.
We can prove Eq~\ref{eq: lemma: guessing r^(t,n)_k1,k2} holds for $t=1$ by induction on $n$.
The proof is omitted.

Then, we give a proof of Lemma~\ref{lemma: guessing r^(t,n)_k1,k2} via Eq~\ref{eq: rank recursion formula} by induction.
Suppose Eq~\ref{eq: lemma: guessing r^(t,n)_k1,k2} holds for $t+n+k+s \le N-1$.
Now consider when $t+n+k+s = N$.
In the following steps, we only need to consider the case $t \ge 2$ and $ t < n-1$.

\vspace{1em}
\noindent
\textbf{Step 2: verify Eq~\ref{eq: lemma: guessing r^(t,n)_k1,k2} when $k=0$ and $1 < t < n-1$.}
\vspace{1em}

If $k=0$, then $k+t-2s<n$ and $r^{(1)}_{t,n,k,t-2s} = -\infty$ and $r^{(2)}_{t,n,k,t-2s} = -\infty$.

\textbf{Case 2.1: If $t = 2s$,} then $r(M^{(t,n)}_{k, k+t-2s}) = r(M^{(t,n)}_{0,0}) = 1$ from definition.
It is easy to verify Eq~\ref{eq: lemma: guessing r^(t,n)_k1,k2} holds in this case.


\textbf{Case 2.2: If $t > 2s$,}
then $(t-1,n-1,k,s)$ are valid and, hence, $r^{(3)}_{t,n,k,t-2s} = r(M^{(t-1,n-1)}_{k, k+(t-1) - 2s})$ can be calculated via Eq~\ref{eq: lemma: guessing r^(t,n)_k1,k2} by induction.

Since $k-s \le 0$, we have
\begin{equation*}
\begin{aligned}
	r^{(3)}_{t,n,k,t-2s} &  = r(M^{(t-1,n-1)}_{k, k+(t-1) - 2s}) \\
	&= \binom{n-1}{k} + D_{t-1,n-1,k,s} \\
	& \le \binom{n}{k} + \sum_{m=k+t-2s+1}^{n} \binom{m-1}{k+t-2s-1}  = \binom{n}{k} + D_{t,k,n,s},
\end{aligned}
\end{equation*}
where the second equality comes from $k-s \le 0$ and Eq~\ref{eq: lemma: guessing r^(t,n)_k1,k2} by induction.

\textbf{Subcase 2.2.1: If $t = n-2$}, then
\begin{equation*}
\begin{aligned}
	r^{(4)}_{t,n,k,t-2s} &  = r(M^{(t,n-1)}_{k, k+t-2s}) + \binom{n-1}{k+t-2s-1}\\
	& = \binom{n-1}{k} + \binom{n-1}{k+t-2s} - 1 + \binom{n-1}{k+t-2s-1} = \binom{n}{n-2s-2} \\
	& = \binom{n}{k} + D_{t,n,k,s},
\end{aligned}
\end{equation*}
where the second equality comes from the trivial situation of Lemma~\ref{lemma: guessing r^(t,n)_k1,k2} which we have proved before
Therefore, we have verified Eq~\ref{eq: lemma: guessing r^(t,n)_k1,k2} in this subcase.

\textbf{Subcase 2.2.2: If $t < n-2$},
then $(t,n-1,k,s)$ are valid and, hence, $r^{(4)}_{t,n,k,t-2s} = r(M^{(t,n-1)}_{k, k+t-2s})$ can be calculated via Eq~\ref{eq: lemma: guessing r^(t,n)_k1,k2} by induction.
\begin{equation*}
\begin{aligned}
	r^{(4)}_{t,n,k,t-2s} &  = r(M^{(t,n-1)}_{k, k+t-2s}) + \binom{n-1}{k+t-2s-1}\\
	& = \binom{n-1}{k} + D_{t,n-1,k,s} + \binom{n-1}{k+t-2s-1} \\
	& = \binom{n}{k} + D_{t,n,k,s},
\end{aligned}
\end{equation*}
where the second equality comes from Step~1 and the third equality comes from $k=0$.

From above, combining Eq~\ref{eq: rank recursion formula}, we can verify that Eq~\ref{eq: lemma: guessing r^(t,n)_k1,k2} holds when $k=0$.

\vspace{1em}
\noindent
\textbf{Step 3: verify Eq~\ref{eq: lemma: guessing r^(t,n)_k1,k2} when $k+t-2s =n$ and $1 < t < n-1$.}
\vspace{1em}

In this case, $r^{(2)}_{t,n,k,t-2s} = -\infty$ and $r^{(4)}_{t,n,k,t-2s} = -\infty$.

\textbf{Case 3.1: If $t = 2s$}, then $k = n$ and
$r(M^{(t,n)}_{k, k+t-2s}) = r(M^{(t,n)}_{n,n}) = 1$ from definition.
It is easy to verify Eq~\ref{eq: lemma: guessing r^(t,n)_k1,k2} holds in this case.

\textbf{Case 3.2: If $t < 2s$}, then $(t-1,n-1,k,s)$ is valid and, hence, $r_{t,n,k,p}^{(3)} =r(M_{k, k+p-1}^{(t-1,n-1)})$ can be calculated via Eq~\ref{eq: lemma: guessing r^(t,n)_k1,k2} by induction.
Since $k-s \ge n-t$, we have
\begin{equation*}
\begin{aligned}
	r^{(3)}_{t,n,k,t-2s} &  = r(M^{(t-1,n-1)}_{k, k+(t-1) - 2s}) \\
	&= \binom{n-1}{k} + D_{t-1,n-1,k,s} \\
	&= \binom{n-1}{k}  \\
	& \le \binom{n}{k} = \binom{n}{k} + D_{t,k,n,s},
\end{aligned}
\end{equation*}
where the second equality comes from Eq~\ref{eq: lemma: guessing r^(t,n)_k1,k2} by induction.



\textbf{Subcase 3.2.1: If $ t =n-2$}, then
\begin{equation*}
	\begin{aligned}
		r^{(1)}_{t,n,k,t-2s} &  = r(M^{(t,n-1)}_{k-1,k-1+t-2s}) + \binom{n-1}{k} \\
		& = \binom{n-1}{k} + \binom{n-1}{k-1} + \binom{n-1}{k-1+t-2s} - 1 \\
		& = \binom{n}{k} = \binom{n}{k} + D_{t,n,k,s},
	\end{aligned}
	\end{equation*}
where the second equality comes from Step 1 and the following equalities come from $k+t-2s=n$.

\textbf{Subcase 3.2.2: If $ t < n-2$}, then $(t,n-1,k-1,s)$ is valid and,
hence, $r^{(1)}_{t,n,k,t-2s} = r(M^{(t,n-1)}_{k-1,k-1+t-2s}) + \binom{n-1}{k}$ can be calculated via Eq~\ref{eq: lemma: guessing r^(t,n)_k1,k2} by induction. That is
\begin{equation*}
\begin{aligned}
	r^{(1)}_{t,n,k,t-2s} &  = r(M^{(t,n-1)}_{k-1,k-1+t-2s}) + \binom{n-1}{k} \\
	& = \binom{n-1}{k-1} + D_{t,n-1,k-1,s} + \binom{n-1}{k}\\
	& = \binom{n}{k} + D_{t,n,k,s},
\end{aligned}
\end{equation*}
where the second equality comes from Eq~\ref{eq: lemma: guessing r^(t,n)_k1,k2} by induction, and the third equality comes from $k+t-2s=n$.

From above, combining Eq~\ref{eq: rank recursion formula}, we can verify that Eq~\ref{eq: lemma: guessing r^(t,n)_k1,k2} holds when $k+t-2s = n$.


\vspace{1em}
\noindent
\textbf{Step 4: verify Eq~\ref{eq: lemma: guessing r^(t,n)_k1,k2} when $k>0$, $k+t-2s<n$, $1 < t < n-1$ and $s>0$.}
\vspace{1em}

In this case, ${\{r^{(i)}_{t,n,k,t-2s}\}}_{1\le i \le 4}$ are all positive.
By the definition, we have that $r^{(3)}_{t,n,k,t-2s} \le r^{(2)}_{t,n,k,t-2s}$ always holds. Hence, we only need to consider $r^{(1)}_{t,n,k,t-2s}$, $r^{(2)}_{t,n,k,t-2s}$ and $ r^{(4)}_{t,n,k,t-2s}$ and their value can be calculated from Eq~\ref{eq: lemma: guessing r^(t,n)_k1,k2} by induction.

\vspace{1em}
\noindent
\textbf{Step 4.1: verify that $r^{(2)}_{t,n,k,t-2s}$ is always equal to RHS of Eq~\ref{eq: lemma: guessing r^(t,n)_k1,k2}.}
\vspace{1em}

Note that $(t-1, n-1, k-1, s-1)$ is valid and, hence, we can calculate $r^{(2)}_{t,n,k,t-2s}$ via Eq~\ref{eq: lemma: guessing r^(t,n)_k1,k2} by induction.

\textbf{Case 4.1.1: If $k-s<0$ or $k-s>n-t$}, then $(k-1)-(s-1) < 0$ or $(k-1)-(s-1) > (n-1)-(t-1)$.

Therefore,
\begin{equation*}
\begin{aligned}
	r^{(2)}_{t,n,k,t-2s} &= r(M_{k-1, (k-1)+(t-1)-2(s-1)}^{(t-1, n-1)}) + \binom{n-1}{k} + \binom{n-1}{k+t-2s-1} \\
	&= \binom{n-1}{k-1} + \binom{n-1}{k} + \binom{n-1}{k+t-2s-1} + D_{t-1, n-1, k-1, s-1}\\
	&= \binom{n}{k} + \sum \limits_{m= k+t-2s+1}^{n} \binom{m-1}{k+t-2s-1} = \binom{n}{k} + D_{t,n,k,s},\\
\end{aligned}
\end{equation*}
where the second equality comes from Eq~\ref{eq: lemma: guessing r^(t,n)_k1,k2} by induction.


\textbf{Case 4.1.2: If $0 \le k-s \le \left\lfloor (n-t)/2 \right\rfloor$}, then $0 \le (k-1)-(s-1) \le \left\lfloor ((n-1)-(t-1))/2 \right\rfloor$.

Therefore,
\begin{equation*}
\begin{aligned}
	r^{(2)}_{t,n,k,t-2s} &= r(M_{k-1, (k-1)+(t-1)-2(s-1)}^{(t-1, n-1)}) + \binom{n-1}{k} + \binom{n-1}{k+t-2s-1} \\
	&= \binom{n}{k} + \sum \limits_{a=0}^{k-s-1} \left( \binom{t-s+2a}{t-s+a-1} - \binom{t-s+2a}{a-1} \right) \\
	& \quad  + \sum \limits_{m = t-3s+1+2k}^{n-s} \left( \binom{m-1}{k+t-2s-1} - \binom{m-1}{k-s-1} \right)  + \sum \limits_{m= n-s+1}^{n} \binom{m-1}{k+t-2s-1}\\
	&= \binom{n}{k} + A^{(1)}_{t,n,k,s} + A^{(2)}_{t,n,k,s} + C_{t,n,k,s},\\
\end{aligned}
\end{equation*}
where he second equality comes from Eq~\ref{eq: lemma: guessing r^(t,n)_k1,k2} by induction.

\textbf{Case 4.1.3: If $ \left\lfloor (n-t)/2 \right\rfloor \le k-s \le n-t $}, then $ \left\lfloor ((n-1)-(t-1))/2 \right\rfloor \le (k-1)-(s-1) \le (n-1)-(t-1) $.

Therefore,
\begin{equation*}
\begin{aligned}
	r^{(2)}_{t,n,k,t-2s} &= r(M_{k-1, (k-1)+(t-1)-2(s-1)}^{(t-1, n-1)}) + \binom{n-1}{k} + \binom{n-1}{k+t-2s-1} \\
	&= \binom{n-1}{k-1} + \binom{n-1}{k} + \binom{n-1}{k+t-2s-1} + B_{t-1,n-1,k-1,s-1} + C_{t-1,n-1,k-1,s-1}\\
	&= \binom{n}{k} + \sum \limits_{a=0}^{n-t-k+s-1} \left( \binom{t-s+2a}{t-s+a-1} - \binom{t-s+2a}{a-1} \right) + \sum \limits_{m= n-s+1}^{n} \binom{m-1}{k+t-2s-1} \\
	&= \binom{n}{k} + B_{t,n,k,s} + C_{t,n,k,s},\\
\end{aligned}
\end{equation*}
where he second equality comes from Eq~\ref{eq: lemma: guessing r^(t,n)_k1,k2} by induction.

Consequently, $r^{(2)}_{t,n,k,t-2s}$ is always equal to RHS of Eq~\ref{eq: lemma: guessing r^(t,n)_k1,k2}.
Hence, we only need to prove $r^{(1)}_{t,n,k,t-2s} \le r^{(2)}_{t,n,k,t-2s}$ and $r^{(4)}_{t,n,k,t-2s} \le r^{(2)}_{t,n,k,t-2s}$.
Then by Eq~\ref{eq: rank recursion formula}, we can finally prove Lemma~\ref{lemma: guessing r^(t,n)_k1,k2}.

\vspace{1em}
\noindent
\textbf{Step 4.2: verify that $r^{(1)}_{t,n,k,t-2s} \le r^{(2)}_{t,n,k,t-2s}$}.
\vspace{1em}


First, it is easy to verify the inequality when $t=n-2$, then we assume $t < n-2$ in the following.

\textbf{Case 4.2.1: If $(k-1)-s < 0$}, then $(k-1)-(s-1) = k-s \le 0$ and
\begin{equation*}
\begin{aligned}
	& D_{t,n-1,k-1,s} - D_{t-1,n-1,k-1,s-1} \\
	=& \sum \limits_{m= k+t-2s}^{n-1} \binom{m-1}{k+t-2s-2} - \sum \limits_{m= k+t-2s+1}^{n-1} \binom{m-1}{k+t-2s-1} \\
	& =  \binom{n-2}{k+t-2s-2} - \sum \limits_{m= k+t-2s+1}^{n-1} \binom{m-2}{k+t-2s-1}.
\end{aligned}
\end{equation*}

Therefore,
\begin{equation}
\begin{aligned}
	 & r^{(1)}_{t,n,k,t-2s} - r^{(2)}_{t,n,k,t-2s} \\
	 =& r(M^{(t, n-1)}_{k-1, k-1+t-2s}) - r(M^{(t-1, n-1)}_{k-1, k-1+t-1-2(s-1)}) - \binom{n-1}{k+t-2s-1}  \\
	 =& -\binom{n-2}{k+t-2s-1} - \sum \limits_{m= k+t-2s+1}^{n-1} \binom{m-2}{k+t-2s-1}  \le 0.
\end{aligned}
\label{eq: appendix A r1 le r2 case 1}
\end{equation}

\textbf{Case 4.2.2: If $ 0 \le (k-1)-s < \lfloor ((n-1)-t)/2 \rfloor $}, then $ 0 \le 1 \le k-s \le \lfloor (n-t-1)/2 \rfloor \le \lfloor (n-t)/2 \rfloor$ and noting that $t-3s+2k-1 \le t-s-1 +(n-t-1) \le n-s-2$
\begin{equation*}
\begin{aligned}
	&A^{(1)}_{t,n-1,k-1,s} - A^{(1)}_{t-1,n-1,k-1,s-1} = - \left( \binom{t+2k-3s-2}{t+k-2s-2} - \binom{t+2k-3s-2}{k-s-2}  \right), \\
	 & A^{(2)}_{t,n-1,k-1,s} - A^{(2)}_{t-1,n-1,k-1,s-1}\\
	=& \sum \limits_{m = t-3s+2k-1}^{(n-1)-s} \left( \binom{m-1}{k+t-2s-2} - \binom{m-1}{k-s-2} \right) \\
&- \sum \limits_{m = t-3s+2k+1}^{n-s} \left( \binom{m-1}{k+t-2s-1} - \binom{m-1}{k-s-1} \right) \\
	=& \left( \binom{t-2s+2k-2}{k+t-2s-2} - \binom{t-3d+2k-2}{k-s-2} \right) - \sum \limits_{m = t-3s+2k}^{n-s-1} \left( \binom{m-1}{k+t-2s-1} - \binom{m-1}{k-s-1} \right),\\
	& C_{t,n-1,k-1,s} - C_{t-1,n-1,k-1,s-1} \\
	=& \sum \limits_{m= n-s}^{n-1} \binom{m-1}{k+t-2s-2} - \sum \limits_{m= n-s+1}^{n-1} \binom{m-1}{k+t-2s-1} \\
	\le & \binom{n-2}{k+t-2s-2} - \sum_{m=n-s}^{n-2}\binom{m-1}{k+t-2s-1}.
\end{aligned}
\end{equation*}
From above three terms, it is easy to get
\begin{equation*}
\begin{aligned}
	& r^{(1)}_{t,n,k,t-2s} - r^{(2)}_{t,n,k,t-2s} \\
	= & r(M^{(t, n-1)}_{k-1, k-1+t-2s}) - r(M^{(t-1, n-1)}_{k-1, k-1+t-1-2(s-1)}) - \binom{n-1}{k-1+t-2s} \\
	= & \sum_{i=1}^{2} (A^{(i)}_{t,n-1,k-1,s} - A^{(i)}_{t-1,n-1,k-1,s-1}) + (C_{t,n-1,k-1,s} - C_{t-1,n-1,k-1,s-1})- \binom{n-1}{k-1+t-2s} \\
	\le & 0.
\end{aligned}
\end{equation*}

\textbf{Case 4.2.3: If $\left\lfloor ((n-1)-t)/2 \right\rfloor \le (k-1)-s \le (n-1)-t $}, then $\left\lfloor (n-t)/2 \right\rfloor \le  \left\lfloor (n-t+1)/2 \right\rfloor \le k-s \le n-t $.

Since $B_{t,n-1,k-1,s} - B_{t-1,n-1,k-1,s-1}=0$, we have,
\begin{equation*}
\begin{aligned}
	& r^{(1)}_{t,n,k,t-2s} - r^{(2)}_{t,n,k,t-2s} \\
		= & r(M^{(t, n-1)}_{k-1, k-1+t-2s}) - r(M^{(t-1, n-1)}_{k-1, k-1+t-1-2(s-1)}) - \binom{n-1}{k-1+t-2s} \\
		=& (B_{t,n-1,k-1,s} - B_{t-1,n-1,k-1,s-1}) + (C_{t,n-1,k-1,s}-C_{t-1,n-1,k-1,s-1})  - \binom{n-1}{k-1+t-2s} \\
		=& C_{t,n-1,k-1,s}-C_{t-1,n-1,k-1,s-1}  - \binom{n-1}{k-1+t-2s} \\
		\le &\binom{n-2}{k+t-2s-2} - \binom{n-1}{k-1+t-2s} - \sum_{m=n-s}^{n-2}\binom{m-1}{k+t-2s-1} \\
		\le& 0 .
\end{aligned}
\end{equation*}
	
\textbf{Case 4.2.4: If $(k-1)-s > (n-1)-t$}, then $k-s > n-t$.

Same as Eq~\ref{eq: appendix A r1 le r2 case 1}, we can verify $r^{(1)}_{t,n,k,t-2s} \le r^{(2)}_{t,n,k,t-2s}$ in this case.

From above all, we have verified $r^{(1)}_{t,n,k,t-2s} \le r^{(2)}_{t,n,k,t-2s}$ in all cases.

\vspace{1em}
\noindent
\textbf{Step 4.3: verify that $r^{(4)}_{t,n,k,t-2s} \le r^{(2)}_{t,n,k,t-2s}$.}
\vspace{1em}




First, if $t=n-2$, then 
\begin{equation*}
\begin{aligned}
	r^{(4)}_{t,n,k,t-2s} &= \binom{n-1}{k+t-2s-1} + r(M^{(t,n-1)}_{k,k+t-2s}) \\ 
	&= \binom{n-1}{k+t-2s-1}  + \binom{n-1}{k} + \binom{n-1}{k+t-2s} -1. 
\end{aligned}
\end{equation*}

Moreover, we have 

\begin{equation*}
\begin{aligned}
	r^{(2)}_{t,n,k,t-2s} &= \binom{n-1}{k} + \binom{n-1}{k+t-2s-1} + r(M^{(t-1,n-1)}_{k-1,k+t-2s})
\end{aligned}
\end{equation*}

If $k-s <= 0$ or $k-s>= n-t=2$, then $r(M^{(t-1,n-1)}_{k-1,k+t-2s}) \ge D_{t-1,n-1,k-1,s-1} = \binom{n-1}{k+t-2s} - 1$. Otherwise, $k-s = 1$ and then $r(M^{(t-1,n-1)}_{k-1,k+t-2s}) \ge \binom{n-1}{k-1} = \binom{n-1}{k+t-2s-1}$. In both cases, we can derive $r^{(4)}_{t,n,k,t-2s} \le r^{(2)}_{t,n,k,t-2s}$. Thus, in the following, we assume $t < n-2$.

\textbf{Case 4.3.1: If $k-s<0$}, then $(k-1)-(s-1) < 0$.

Therefore,
\begin{equation}
\label{eq: case 4.3.1}
\begin{aligned}
	& r^{(4)}_{t,n,k,t-2s} - r^{(2)}_{t,n,k,t-2s} \\
	=	& r(M^{(t, n-1)}_{k, k+t-2s}) - r(M^{(t-1, n-1)}_{k-1, k-1+t-1-2(s-1)}) - \binom{n-1}{k} \\
	=   & \binom{n-1}{k} - \binom{n-1}{k-1} - \binom{n-1}{k} + (D_{t,n-1,k,s} - D_{t-1,n-1,k-1,s-1}) \\
	\le &  \sum \limits_{m= k+t-2s+1}^{n-1} \binom{m-1}{k+t-2s-1} - \sum \limits_{m= k+t-2s+1}^{n-1} \binom{m-1}{k+t-2s-1} = 0.
\end{aligned}
\end{equation}

\textbf{Case 4.3.2: If $0 \le k-s \le \left\lfloor ((n-1)-t)/2 \right\rfloor$}, then $0 \le k-s \le \left\lfloor (n-t-1)/2 \right\rfloor \le \left\lfloor (n-t)/2 \right\rfloor$ and $2(k-s) \le n-t-1$. We can derive that

\begin{equation*}
\begin{aligned}
	& A^{(1)}_{t,n-1,k,s} - A^{(1)}_{t-1,n-1,k-1,s-1} = 0, \\
	& A^{(2)}_{t,n-1,k,s} - A^{(2)}_{t-1,n-1,k-1,s-1} = - \left(\binom{n-s-1}{k+t-2s-1} - \binom{n-s-1}{k-s-1} \right) , \\
	& C_{t,n-1,k,s} - C_{t-1,n-1,k-1,s-1} \\
	= & \sum \limits_{m= n-s}^{n-1} \binom{m-1}{k+t-2s-1} - \sum \limits_{m= n-s+1}^{n-1} \binom{m-1}{k+t-2s-1} \\
	=& \binom{n-s-1}{k+t-2s-1}.
\end{aligned}
\end{equation*}

Therefore,
\begin{equation*}
\begin{aligned}
	& r^{(4)}_{t,n,k,t-2s} - r^{(2)}_{t,n,k,t-2s} \\
	=	& r(M^{(t, n-1)}_{k, k+t-2s}) - r(M^{(t-1, n-1)}_{k-1, k-1+t-1-2(s-1)}) - \binom{n-1}{k} \\
	=   & - \binom{n-1}{k-1} + \sum_{i=1}^{2} (A^{(i)}_{t,n-1,k,s} - A^{(i)}_{t-1,n-1,k-1,s-1}) + (C_{t,n-1,k,s} - C_{t-1,n-1,k-1,s-1}) \\
	= 	& \binom{n-s-1}{k-s-1} - \binom{n-1}{k-1} \le 0.
\end{aligned}
\end{equation*}

\textbf{Case 4.3.3: If $\left\lfloor ((n-1)-t)/2 \right\rfloor < k-s \le (n-1)-t $}, then $\left\lfloor (n-t)/2 \right\rfloor \le \left\lfloor (n-t+1)/2 \right\rfloor \le k-s \le n-t$.

We have
\begin{equation*}
\begin{aligned}
	& B_{t,n-1,k,s} - B_{t-1,n-1,k-1,s-1} \\
	= & \sum \limits_{a=0}^{n-t-k+s-2} \left( \binom{t-s+2a}{t-s+a-1} - \binom{t-s+2a}{a-1} \right) - \sum \limits_{a=0}^{n-t-k+s-1} \left( \binom{t-s+2a}{t-s+a-1} - \binom{t-s+2a}{a-1} \right)\\
 &= -\binom{2n-2k-t+s-2}{n-k-2} + \binom{2n-2k-t+s-2}{n-t-k+s-2}.
\end{aligned}
\end{equation*}

Write $\lambda \triangleq k-s - \frac{n-t}{2}$, then $\lambda \ge 0$ by the condition.
Therefore,
\begin{equation*}
\begin{aligned}
	& r^{(4)}_{t,n,k,t-2s} - r^{(2)}_{t,n,k,t-2s} \\
	=	& r(M^{(t, n-1)}_{k, k+t-2s}) - r(M^{(t-1, n-1)}_{k-1, k-1+t-1-2(s-1)}) - \binom{n-1}{k} \\
	=   & - \binom{n-1}{k-1} +  (B_{t,n-1,k,s} - B_{t-1,n-1,k-1,s-1}) + (C_{t,n-1,k,s} - C_{t-1,n-1,k-1,s-1}) \\
	\le	& \binom{n-s-1}{k+t-2s-1} - \binom{n-1}{k-1} -\binom{2n-2k-t+s-2}{n-k-2} + \binom{2n-2k-t+s-2}{n-t-k+s-2} \\ 
	\le & \binom{n-s-1}{k-s-2\lambda} - \binom{n-1}{k-1} - \binom{n-s-2-2\lambda}{k-s-2\lambda} + \binom{n-s-2-2\lambda}{k-s-2\lambda -2}.
\end{aligned}
\end{equation*}

If $\lambda = 0$, it is easy to verify $r^{(4)}_{t,n,k,t-2s} - r^{(2)}_{t,n,k,t-2s} \le 0$ with the above equation. If $\lambda \ge 1$, then we first claim that $\binom{n-s-1}{k-s-2\lambda} \le \binom{n-s-1}{k-s-1}$.
If $k-s-1 \le \frac{n-s-1}{2}$, it holds naturally.
Otherwise, when $k-s-1 > \frac{n-s-1}{2}$, note that $k-s-2\lambda \le \frac{n-s-1}{2}$ holds, since $t \ge 2s \ge 2$. It suffices to prove that $k-s-1 - \frac{n-s-1}{2} \le \frac{n-s-1}{2} - (k-s-2\lambda )$ which is easy to verify.
As a consequence, we have 

\begin{equation*}
\begin{aligned}
	& r^{(4)}_{t,n,k,t-2s} - r^{(2)}_{t,n,k,t-2s} \\
	\le &\binom{n-s-1}{k+t-2s-1} - \binom{n-1}{k-1} \\
	=& \binom{n-s-1}{k-s-2\lambda} - \binom{n-1}{k-1} \\
	\le & \binom{n-s-1}{k-s-1} - \binom{n-1}{k-1} \le 0.
\end{aligned}
\end{equation*}

\textbf{Case 4.3.4: If $k-s > (n-1)-t$}, then $k-s \ge n-t$.

Same as Eq~\ref{eq: case 4.3.1}, we can verify $r^{(4)}_{t,n,k,t-2s} \le r^{(2)}_{t,n,k,t-2s}$ in this case.

From above all, we have verified $r^{(4)}_{t,n,k,t-2s} \le r^{(2)}_{t,n,k,t-2s}$ in all cases.

\vspace{1em}
\noindent
\textbf{Step 5: verify Eq~\ref{eq: lemma: guessing r^(t,n)_k1,k2} when $k>0$, $k+t-2s<n$ and $s=0$.} 
In this case, $r^{(2)}_{t,n,k,t-2s} = -\infty$ and ${\{r^{(i)}_{t,n,k,t-2s}\}}_{1=1,3,4}$ are positive.

\textbf{Case 5.1: If $1 \le k < \lfloor (n-t)/2 \rfloor$}, then we have $k \le \lfloor ((n-1)-t)/2 \rfloor$ and $2k \le n-t-1$.
We first verify that $r^{(4)}_{t,n,k,t-2s}$ is always equal to RHS of Eq~\ref{eq: lemma: guessing r^(t,n)_k1,k2}.

\begin{equation*}
\begin{aligned}
	r^{(4)}_{t,n,k,t-2s} &= \binom{n-1}{k+t-2s-1} + r(M^{(t,n-1)}_{k, k+t-2s}) \\ 
	&= \binom{n-1}{k+t-2s-1} + \binom{n-1}{k} + A^{(1)}_{t,n-1,k,s}+ A^{(2)}_{t,n-1,k,s} + C_{t,n-1,k,s}.
\end{aligned}
\end{equation*}

Note that $A^{(1)}_{t,n-1,k,s} = A^{(1)}_{t,n,k,s}$, $C^{(1)}_{t,n-1,k,s} = 0 = C_{t,n,k,s}$ when $s=0$, and

\begin{equation*}
\begin{aligned}
	A^{(2)}_{t,n-1,k,s} & =  A^{(2)}_{t,n,k,s} - \left( \binom{n-1}{k+t-2s-1} - \binom{n-1}{k-s-1} \right).
\end{aligned}
\end{equation*}

Then we can derive that 

\begin{equation*}
\begin{aligned}
	r^{(4)}_{t,n,k,t-2s} &=  \binom{n}{k} + A^{(1)}_{t,n,k,s}+ A^{(2)}_{t,n,k,s} + C_{t,n,k,s}.
\end{aligned}
\end{equation*}

Consequently, $r^{(4)}_{t,n,k,t-2s}$ is always equal to RHS of Eq~\ref{eq: lemma: guessing r^(t,n)_k1,k2} and we then only need to prove $r^{(1)}_{t,n,k,t-2s} \le r^{(4)}_{t,n,k,t-2s}$ and $r^{(3)}_{t,n,k,t-2s} \le r^{(4)}_{t,n,k,t-2s}$. We omit the details here.

\vspace{.2cm}
\textbf{Case 5.2: If $\lfloor (n-t)/2 \rfloor < k \le n-t $}, then $ \lfloor (n-1-t)/2 \rfloor \le k-1 \le (n-1) -t$.
We first verify that $r^{(1)}_{t,n,k,t-2s}$ is always equal to RHS of Eq~\ref{eq: lemma: guessing r^(t,n)_k1,k2}.

\begin{equation*}
\begin{aligned}
	r^{(1)}_{t,n,k,t-2s} &= \binom{n-1}{k} + r(M^{(t,n-1)}_{k-1, k+t-2s-1}) \\ 
	&= \binom{n-1}{k} + \binom{n-1}{k-1} + B_{t,n-1,k-1,s} + C_{t,n-1,k-1,s} \\ 
	& = \binom{n}{k} + B_{t,n,k,s} + C_{t,n,k,s}. \\ 
\end{aligned}
\end{equation*}

Consequently, $r^{(1)}_{t,n,k,t-2s}$ is always equal to RHS of Eq~\ref{eq: lemma: guessing r^(t,n)_k1,k2} and we then only need to prove $r^{(3)}_{t,n,k,t-2s} \le r^{(1)}_{t,n,k,t-2s}$ and $r^{(4)}_{t,n,k,t-2s} \le r^{(1)}_{t,n,k,t-2s}$. We omit the details here.

\vspace{1em}
\noindent
\textbf{Step 6: verify the base case of induction.}
\vspace{1em}

Consider when $t+n+k+s=2$ and show that $(t,n,k,s)$ is valid.
The only case is $t=1,n=1,k=0,s=0$ and, then, by Step 1, we can verify Eq~\ref{eq: lemma: guessing r^(t,n)_k1,k2} in this case.

By induction, we have finally proved Lemma~\ref{lemma: guessing r^(t,n)_k1,k2}.  
\hfill$\square$


\section*{Appendix B:\@ Proof of Lemma~\ref{lemma: bandwidth guessing M^{(t,n)}}}

We give the proof of Lemma~\ref{lemma: bandwidth guessing M^{(t,n)}} via Eq~\ref{eq: bandwidth recursion formula}.
Note that when $t=1$, $bw(M^{(1,n)}) = \sum_{m=0}^{n-1}\binom{m}{\lfloor m/2 \rfloor}$~\cite{pw_Hypercube_harper1966optimal}.

If $n = 2l+1$ for some integer $l$, then
\begin{equation*}
\begin{aligned}
	& \sum_{m=0}^{n-1} \binom{m}{\lfloor m/2 \rfloor} \\
	=& \sum_{a=0}^{l} \binom{2a}{a} + \sum_{a=0}^{l-1}\binom{2a+1}{a} \\
	=& \binom{2l+1}{l} + \sum_{a=0}^{l-1} \left( \binom{2a+1}{a} - \binom{2a+1}{a-1} \right)\\
	=& \sum_{k = \left\lfloor (n-t)/2 \right\rfloor }^{\left\lfloor (n-t)/2 \right\rfloor +t -1} \binom{n}{k} + \sum_{a=0}^{\left\lfloor (n-t-1)/2 \right\rfloor } \left( \binom{t+2a}{t+a-1} - \binom{t+2a}{a-1} \right) ,
\end{aligned}
\end{equation*}
where the second equality can be proved by induction on $l$.

If $n = 2l$ for some integer $l$, then
\begin{equation*}
\begin{aligned}
	& \sum_{m=0}^{n-1} \binom{m}{\lfloor m/2 \rfloor} \\
	=& \sum_{a=0}^{l-1} \binom{2a}{a} + \sum_{a=0}^{l-1}\binom{2a+1}{a} \\
	=& \binom{2l}{l} + \sum_{a=0}^{l-1} \left( \binom{2a+1}{a} - \binom{2a+1}{a-1} \right)\\
	=& \sum_{k = \left\lfloor (n-t)/2 \right\rfloor }^{\left\lfloor (n-t)/2 \right\rfloor +t -1} \binom{n}{k} + \sum_{a=0}^{\left\lfloor (n-t-1)/2 \right\rfloor } \left( \binom{t+2a}{t+a-1} - \binom{t+2a}{a-1} \right) ,
\end{aligned}
\end{equation*}
where the second equality can be proved by induction on $l$ as well.

Therefore, when $t=1$, Eq~\ref{eq: bandwidth guessing M^{(t,n)}} holds. Then we intend to prove Lemma~\ref{lemma: bandwidth guessing M^{(t,n)}} by induction on $t$.
Suppose Eq~\ref{eq: bandwidth guessing M^{(t,n)}} holds for $t < T$, now consider when $t=T$.

First consider the situation when $k= \left\lfloor (n-t)/2 \right\rfloor $ and $p=t$. Then we have
\begin{equation}
	\label{eq: appendix value of bw(M^{(t,n)})}
\begin{aligned}
	\sum_{q=1}^{t-1} \binom{n}{k+q} + r(M^{(t,n)}_{k,k+t}) &= \sum_{q=0}^{t-1} \binom{n}{k+q} +  \sum_{a=0}^{\left\lfloor (n-t-1)/2 \right\rfloor } \left( \binom{t+2a}{t+a-1} - \binom{t+2a}{a-1}. \right)
\end{aligned}
\end{equation}
The result matches the RHS of Eq~\ref{eq: bandwidth guessing M^{(t,n)}}. In the following, we only need to prove other term is no more than this value.

For convenience, define $\widetilde{r} (M^{(t,n)}_{k,k+p}) = \sum_{q = 1}^{p-1} \binom{n}{k+q} + r(M^{(t,n)}_{k,k+p})$ when $1 \le p \le t$ and $\widetilde{r} (M^{(t,n)}_{k,k}) = bw(M^{(t,n)}_{k,k}) $.
Actually, $\widetilde{r} (M^{(t,n)}_{k,k+p})$ where $0 \le k \le k+p \le n$ is exactly the maximum Manhattan distance from a nonzero element of $M^{(t,n)}_{k,k+p}$ in $M$ to the main diagonal of the matrix $M$.
There exists a consist expression for $\widetilde{r} (M^{(t,n)}_{k,k+p})$ as following.
\begin{equation*}
	\widetilde{r} (M^{(t,n)}_{k,k+p}) = \sum_{q=0}^{p-1} \binom{n}{k+q} - \binom{n}{k} + r(M^{(t,n)}_{k,k+p}).
\end{equation*}

Then, our purpose is to prove $\widetilde{r} (M^{(t,n)}_{k,k+p})$ is no larger than Eq~\ref{eq: appendix value of bw(M^{(t,n)})} for any integer $k,p$ satisfying $0 \le k \le k+p \le n$ and $0 \le p \le t$.

From definition, when $k = 0$, we have that $\widetilde{r}(M^{(t,n)}_{k,k+p})$ reach its maximal when $p=t$.
Similarly, when $k+p=n$, $\widetilde{r}(M^{(t,n)}_{k,k+p})$ reach its maximal when $k = n-t$, that is, $p=t$.
When $k>0$ and $k+p<n$, we have $\widetilde{r}(M^{(t,n)}_{k,k+p}) \le \widetilde{r}(M^{(t,n)}_{k-1,k+p+1})$ when $p + 2 \le t$.
Hence, we only need to prove $\widetilde{r} (M^{(t,n)}_{k,k+p})$ is no larger than Eq~\ref{eq: appendix value of bw(M^{(t,n)})} for the following two cases:

\begin{enumerate}[(1)]
	\item $p = t-1$,
	\item $p = t$.
\end{enumerate}



\textbf{Case 1: } $p = t-1$. In this case, we have $ \widetilde{r} (M^{(t,n)}_{k,k+p}) = \widetilde{r} (M^{(t-1,n)}_{k,k+p}) \le bw(M^{(t-1,n)})$ by definition.
Since the value of $bw(M^{(t-1,n)})$ can be calculated from Lemma~\ref{lemma: bandwidth guessing M^{(t,n)}} by induction, and it is not hard to verify the RHS of Eq.~\ref{eq: bandwidth guessing M^{(t,n)}} is increasing with respect to $t$.
Then we can reach our conclusion in this case.

\textbf{Case 2: } $p = t$.

If $k \le \left\lfloor (n-t)/2 \right\rfloor$, then
\begin{equation*}
\begin{aligned}
	&\widetilde{r} (M^{(t,n)}_{k,k+t}) - \widetilde{r} (M^{(t,n)}_{k-1,k+t-1}) \\
	=& \binom{n}{k+t-1} - \binom{n}{k} + \binom{n}{k} - \binom{n}{k-1} +  \left( \binom{2k+t-2}{k+t-2} - \binom{2k+t-2}{k-2} \right) \\
	 &+ \sum \limits_{m = t+1+2k}^{n-s} \left( \binom{m-1}{k+t-1} - \binom{m-1}{k-1} \right) - \sum \limits_{m = t-1+2k}^{n-s} \left( \binom{m-1}{k+t-2} - \binom{m-1}{k-2} \right) \\
	=& \binom{n}{k+t-1} - \binom{n}{k-1} + \sum \limits_{m = t+2k}^{n-1} \left( \binom{m-1}{k+t-2} - \binom{m-1}{k-2} \right) - \left( \binom{n-1}{k+t-2} - \binom{n-1}{k-2} \right) \\
	=& \binom{n-1}{k-t-1} - \binom{n-1}{k-1} + \sum \limits_{m = t+2k}^{n-1} \left( \binom{m-1}{k+t-2} - \binom{m-1}{k-2} \right) \\
	\geq & 0.
\end{aligned}
\end{equation*}

It shows that when $k \le \left\lfloor (n-t)/2 \right\rfloor$, $\widetilde{r} (M^{(t,n)}_{k,k+t})$ reaches its maximum at $k = \left\lfloor (n-t)/2 \right\rfloor$.
If $k \geq \left\lfloor (n-t)/2 \right\rfloor$, then
\begin{equation*}
\begin{aligned}
	&\widetilde{r} (M^{(t,n)}_{k,k+t}) - \widetilde{r} (M^{(t,n)}_{k+1,k+t+1}) \\
	=& \binom{n}{k+1} - \binom{n}{k+t} + \binom{n}{k} - \binom{n}{k+1} + \left( \binom{2n-t-2k-4}{n-k-3} - \binom{2n-t-2k-4}{n-t-k-3}  \right) \\
	\geq & 0.
\end{aligned}
\end{equation*}
Combining the above situations, we have proved that $\widetilde{r} (M^{(t,n)}_{k,k+t})$ reaches its maximum at $k = \left\lfloor (n-t)/2 \right\rfloor$.
That is, the maximum value of is $\widetilde{r} (M^{(t,n)}_{k,k+t})$ as Eq~\ref{eq: appendix value of bw(M^{(t,n)})}.

From above all, we have proved Lemma~\ref{lemma: bandwidth guessing M^{(t,n)}}.  
\hfill$\square$

\end{document}